\documentclass[11pt,twoside]{article}

\usepackage{amsmath,amssymb,amsthm,mathtools}
\usepackage[a4paper,margin=1in]{geometry}
\usepackage{microtype}
\usepackage{array,needspace}
\allowdisplaybreaks[2]
\newcolumntype{L}[1]{>{\raggedright\arraybackslash}p{#1}}
\usepackage[hidelinks]{hyperref}

\numberwithin{equation}{section}

\newtheorem{theorem}{Theorem}[section]
\newtheorem{proposition}[theorem]{Proposition}
\newtheorem{lemma}[theorem]{Lemma}
\newtheorem{corollary}[theorem]{Corollary}
\theoremstyle{remark}
\newtheorem{remark}[theorem]{Remark}

\begin{document}

\title{Strong Solutions and Long-Time Behavior for One-Dimensional Compressible Navier--Stokes Equations with Thermodynamically Compatible Pressure Laws}
\author{Qingsong Zhao$^{a}$\thanks{Corresponding author. E-mail addresses: \href{mailto:qqsszhao@nyist.edu.cn}{qqsszhao@nyist.edu.cn} (Q. Zhao), \href{mailto:boran_zhu@outlook.com}{boran\_zhu@outlook.com} (B. Zhu).},\quad Boran Zhu$^{b}$\\[0.45em]
{\small\itshape $^{a}$ School of Mathematics and Science, Nanyang Institute of Technology, Nanyang 473004, China.}\\
{\small\itshape $^{b}$ School of Mathematics and Statistics, Wuhan University of Technology, Wuhan 430070, China.}}
\date{}
\maketitle

\begin{abstract}
Motivated in part by non-ideal equations of state such as the van der
Waals law, we study the Cauchy problem for the one-dimensional full
compressible Navier--Stokes equations with constant positive viscosity
and heat conductivity and general thermodynamically compatible
pressure and internal-energy functions. Assuming $e_\theta>0$, we obtain
local $H^1(\mathbb R)$ strong solutions for initial data with positive
specific volume and temperature, without sign conditions on the
pressure derivatives. Near an equilibrium satisfying
$p_v(\bar v,\bar\theta)<0$, sufficiently small $H^1(\mathbb R)$ perturbations
generate global strong solutions with uniform thermodynamic bounds
and time-integrable dissipation. For the full constitutive class,
all first spatial derivatives decay in $L^2(\mathbb R)$ and the perturbation
converges to zero in $L^p(\mathbb R)$ for every $2<p\le\infty$.
We also identify a locally entropy-affine subclass of constitutive laws for which a Fourier--Duhamel argument establishes convergence to equilibrium in $H^1(\mathbb R)$. For more general thermodynamically compatible constitutive laws, we establish a conditional criterion for convergence in $H^1(\mathbb R)$, based on the space-time integrability of the nonlinear pressure remainder in entropy variables. 

\\[2mm]
\end{abstract}

\section{Introduction}\label{sec:intro}

In this paper, we consider the one-dimensional compressible non-isentropic Navier--Stokes equations describing the motion of a viscous and heat-conducting fluid. In Lagrangian mass coordinates, the system takes the form
\begin{subequations}\label{NS}
\begin{align}
v_t-u_x &=0, \label{NSa}\\
u_t+p_x &=\left(\frac{\mu u_x}{v}\right)_x, \label{NSb}\\
\left(e+\frac12u^2\right)_t+(up)_x
&=\left(\frac{\kappa\theta_x}{v}\right)_x
+\left(\frac{\mu uu_x}{v}\right)_x. \label{NSc}
\end{align}
\end{subequations}
Here $v(t,x)>0$, $u(t,x)$, $\theta(t,x)>0$, and $e(t,x)$ denote the specific volume, fluid velocity, absolute temperature, and specific internal energy, respectively, while $p(t,x)$ denotes the pressure. Throughout the paper, the viscosity $\mu>0$ and the heat conductivity $\kappa>0$ are assumed to be positive constants.

We study the Cauchy problem for \eqref{NS} with initial data
\begin{equation}\label{initial}
(v,u,\theta)(0,x)=(v_0(x),u_0(x),\theta_0(x)),
\qquad x\in\mathbb{R},
\end{equation}
subject to the far-field condition
\begin{equation}\label{far-field}
\lim_{x\to\pm\infty}
(v_0(x),u_0(x),\theta_0(x))
=(\bar v,0,\bar\theta),
\end{equation}
where $\bar v>0$ and $\bar\theta>0$ are fixed constants. We set
\begin{equation}\label{bar-p}
\bar p:=p(\bar v,\bar\theta).
\end{equation}

For an ideal polytropic gas, the equations of state are
\[
p(v,\theta)=\frac{R\theta}{v},
\qquad
e(v,\theta)=C_v\theta,
\]
where $R>0$ is the gas constant and $C_v>0$ is the constant specific heat at fixed volume. Such a constitutive law possesses a particularly simple thermodynamic structure and has played a central role in the classical mathematical theory. A concrete motivation is the van der Waals equation of state,
$p=R\theta/(v-b)-a/v^2$, defined for $v>b$ and exhibiting competing
repulsive and attractive terms. Its pressure derivative need not have
a fixed sign on the physical state space. Moreover, a consistent
internal energy must be prescribed together with the pressure: a
natural choice is $e=E(\theta)-a/v$, with $E'(\theta)>0$.
Our analysis concerns small perturbations of a \emph{mechanically
stable} state, not the entire phase diagram or phase transitions.
Section~\ref{sec:constitutive-models} retains the complete Gibbs
calculation, mechanical-stability test, and localization needed to
apply the abstract theorems to the singular van der Waals formula.

For other non-ideal fluids, pressure and internal energy may likewise
have more complicated coupled dependence on specific volume and
temperature. This motivates us to consider general constitutive functions
\begin{equation}\label{assumption_pe}
p,e\in C^2(\mathbb{R}_+\times\mathbb{R}_+),
\qquad
e_\theta(v,\theta)>0
\quad\text{for all }(v,\theta)\in\mathbb{R}_+\times\mathbb{R}_+.
\end{equation}
We write
\begin{equation}\label{specific-heat}
c(v,\theta):=e_\theta(v,\theta)>0,
\end{equation}
for the specific heat at constant volume. In contrast with the ideal-gas model, $c(v,\theta)$ is not assumed to be constant.

The pressure and the internal energy are required to satisfy the standard thermodynamic compatibility relation
\begin{equation}\label{thermo-relation}
e_v(v,\theta)
=
\theta p_\theta(v,\theta)-p(v,\theta).
\end{equation}
Equivalently, there exists, up to an additive constant, a specific entropy $s=s(v,\theta)$ satisfying the Gibbs relation
\begin{equation}\label{Gibbs}
\theta\,\mathrm{d}s=\mathrm{d}e+p\,\mathrm{d}v.
\end{equation}
Indeed, \eqref{thermo-relation} implies
\begin{equation}\label{entropy-derivatives}
s_v(v,\theta)=p_\theta(v,\theta),
\qquad
s_\theta(v,\theta)=\frac{e_\theta(v,\theta)}{\theta}.
\end{equation}
The compatibility of these two identities follows from
\[
\frac{\partial}{\partial\theta}p_\theta
=p_{\theta\theta}
=
\frac{\partial}{\partial v}
\left(\frac{e_\theta}{\theta}\right),
\]
which is a direct consequence of \eqref{thermo-relation}.

Using \eqref{NSa} and \eqref{thermo-relation}, the total-energy equation \eqref{NSc} can be equivalently written as the temperature equation
\begin{equation}\label{temperature-equation}
c(v,\theta)\theta_t
+\theta p_\theta(v,\theta)u_x
=
\left(\frac{\kappa\theta_x}{v}\right)_x
+\frac{\mu u_x^2}{v}.
\end{equation}
Consequently, the positivity of $e_\theta$ is precisely the non-degeneracy condition for the time derivative in the temperature equation. Notice that no assumption that $e_\theta$ is constant is needed.

\paragraph{Notation.}
Unless otherwise indicated, all spatial norms are taken over
$\mathbb R$. For $1\le p\le\infty$ and $m\ge0$, we write
$\|f\|_p=\|f\|_{L^p}$ and
$\|f\|_{H^m}=\|f\|_{H^m(\mathbb R)}$; in particular,
$\|f\|=\|f\|_2=\|f\|_{L^2}$ and
$\|f\|_\infty=\|f\|_{L^\infty}$.
For a time interval $I$ and a spatial function space $X$, we abbreviate
$\|f\|_{L^q_tX_x}=\|f\|_{L^q(I;X)}$, where $I$ is understood
from the context; we use $\|f\|_{L^q(a,b;X)}$ when the interval
is specified explicitly. For vector-valued functions, we use
standard product-space norms with the Euclidean norm on component
vectors.

\subsection{Related works and the scope of the results}

The local theory for general compressible fluids goes back to Nash
\cite{Nash1962}; see also Tani \cite{Tani1977} for initial--boundary
value problems. Classical one-dimensional global results include
Kazhikhov and Shelukhin \cite{Kazhikhov1977}, Matsumura and Nishida
\cite{Matsumura1980}, and Kawashima and Nishida \cite{Kawashima1981}.
For unbounded domains and ideal-gas constitutive laws, Jiang
\cite{Jiang1999} studied uniform estimates and large-time behavior.

Non-ideal state equations have also been investigated extensively.
Kawohl \cite{Kawohl1985} and Jiang \cite{Jiang1994a,Jiang1994b}
established important existence and asymptotic results for
one-dimensional real gases. Other boundary-value settings include
the outer-pressure problem of Luo \cite{Luo1997} and the fixed-endpoint
problem of Qin \cite{Qin2002}. Ducomet and Zlotnik
\cite{DucometZlotnik2002,DucometZlotnik2003} obtained stabilization
results for a class of nonmonotone state laws of the form
$p(v,\theta)=p_0(v)+p_1(v)\theta$. Thus nonmonotone pressure and
long-time stability are not new topics in themselves. Indeed, the
classical van der Waals pressure itself has this temperature-affine
form, with $p_0(v)=-a/v^2$ and $p_1(v)=R/(v-b)$ on $v>b$;
it therefore cannot \emph{by itself} serve as evidence that the
present constitutive class extends all earlier nonmonotone-pressure
results. The independent smooth example in
Section~\ref{sec:constitutive-models}, which contains a nonzero
$\theta^2$ term and a nonconstant specific heat, illustrates a
different temperature dependence permitted by the current hypotheses.
In the Cauchy
setting, Qin, Wu, and Liu \cite{QinWuLiu2007} studied global Sobolev
regularity and large-time behavior for a polytropic ideal gas, while
Qin's survey \cite{Qin2018Survey} documents a broader one-dimensional
real-gas theory. Neither the whole-line formulation nor the mere use
of $H^1$ spaces constitutes a priority claim here. The precise
constitutive and regularity hypotheses of the older real-gas results
must be considered before asserting that one theorem extends another.

Complementary work treats variable transport coefficients
\cite{JenssenKarper2010,LiuYangZhaoZou2014,DongGuo2025}, far-field
vacuum \cite{LiXin2022}, and ionized-gas thermodynamics
\cite{LiaoXiongZhao2023}. Of particular relevance, Chen, Gui, and
Jiang \cite{ChenGuiJiang2026} proved local and small-perturbation
global well-posedness for general pressure laws in higher dimensions,
using a Besov-space framework. Their equilibrium condition
$P_\rho(\bar\rho,\bar\theta)>0$ corresponds, under $\rho=1/v$, to
$p_v(\bar v,\bar\theta)<0$, since
\[
 P_\rho(\bar\rho,\bar\theta)
 =-\bar v^{\,2}p_v(\bar v,\bar\theta).
\]
The present setting differs in dimension, coordinates, solution space,
and energy method; it is not claimed that general-pressure small-data
well-posedness was previously unavailable. More generally, the existing
literature already treats one-dimensional real gases, nonmonotone
pressure laws and several kinds of asymptotic stability; neither
nonmonotonicity nor the Cauchy formulation is presented here as new.

For clarity, the relation between the cited works and the present
results is summarized in Table~\ref{tab:prior-work-comparison}.
The comparison is deliberately restricted to the problem classes and
conclusions identified in the cited literature above; it does not
assert a strict inclusion of the hypotheses of any earlier theorem.

\begin{table}[!t]
\centering
\caption{Comparison of selected related works with the present results.}
\label{tab:prior-work-comparison}
\small
\setlength{\tabcolsep}{3pt}
\renewcommand{\arraystretch}{1.08}
\begin{tabular}{@{}L{0.18\textwidth}L{0.35\textwidth}L{0.41\textwidth}@{}}
\hline
\textbf{Works} & \textbf{Problem class} & \textbf{Relation to the present paper}\\
\hline
Jiang~\cite{Jiang1999}; Qin et al.~\cite{QinWuLiu2007}
& Jiang: large-time behavior in unbounded domains for a one-dimensional polytropic ideal gas; Qin et al.: the one-dimensional ideal-gas Cauchy problem.
& These provide precedents for unbounded-domain asymptotics and ideal-gas Cauchy theory, but not for the general Gibbs-compatible pair $(p,e)$ considered here.\\[0.6mm]
Kawohl~\cite{Kawohl1985}; Jiang~\cite{Jiang1994a,Jiang1994b}; Qin~\cite{Qin2002}
& One-dimensional real-gas existence and asymptotic problems, including initial--boundary value settings with problem-specific boundary conditions.
& They establish a substantial real-gas existence and stability literature; those boundary-value formulations differ from the whole-line Cauchy problem studied here.\\[0.6mm]
Ducomet and Zlotnik~\cite{DucometZlotnik2002,DucometZlotnik2003}
& Stabilization for nonmonotone heat-conducting state laws, including $p(v,\theta)=p_0(v)+p_1(v)\theta$.
& Nonmonotonicity and stabilization are not claimed as new. The smooth example here also permits a nonzero $\theta^2$ term and state-dependent specific heat.\\[0.6mm]
Chen, Gui, and Jiang~\cite{ChenGuiJiang2026}
& Three-dimensional compressible Navier--Stokes--Fourier equations with general pressure laws; local and small-data global well-posedness in Besov spaces.
& The present work treats the one-dimensional Lagrangian $H^1$ Cauchy problem with a different energy method; no priority claim is made for general-pressure well-posedness.\\[0.6mm]
\textbf{Present paper}
& General Gibbs-compatible $C^2$ pairs $(p,e)$ with $e_\theta>0$, near an equilibrium with $p_v<0$.
& Global strong solutions and gradient decay; full $H^1$ convergence under the entropy-affine structure or the stated pressure-remainder integrability condition.\\
\hline
\end{tabular}
\end{table}

Accordingly, the contribution claimed here is the specific combination
of a one-dimensional $H^1$ theory under the Gibbs identity and locally
positive specific heat, global estimates near a mechanically stable
state, and an entropy-variable criterion that separates gradient
relaxation from full $H^1$ relaxation. This statement describes the
results proved in this paper; it is not a priority claim over the
previous real-gas existence or stabilization literature.

The purpose of this paper is to give a direct $H^1$ strong-solution
analysis of the one-dimensional Cauchy system under the Gibbs
compatibility relation, allowing state-dependent specific heat. The
local theorem uses only positive $e_\theta$ and no pressure-derivative
sign condition. Near a mechanically stable equilibrium, the relative
energy and effective-gradient estimates yield small-data global
solvability and decay of spatial gradients, hence convergence in
$L^p$, $p>2$. A separate entropy-coordinate Fourier--Duhamel argument
establishes full $H^1$ convergence for a restricted entropy-affine
class, or conditionally when a specified pressure remainder is
space--time integrable. These are different conclusions with different
hypotheses; no general-pressure $L^2$ convergence follows merely from
the gradient estimates. The entropy-affine condition is a restrictive
sufficient condition for the stronger theorem, not a generic property
of classical gases. In particular, the van der Waals motivation leads
to the general-law theorems, not automatically to entropy-affine
$H^1$ relaxation. Its detailed constitutive calculations and a
second globally smooth nonmonotone example are retained in
Section~\ref{sec:constitutive-models}.

\subsection{Main results and relative-energy structure}

We first state the local existence result. At the local level, the essential thermodynamic assumption is the positivity of $e_\theta$, which ensures the nondegeneracy of the temperature equation on compact thermodynamic regions. No sign condition on $p_v$ or $p_\theta$ is required.

\begin{theorem}[Local existence]\label{thm:local-large}
Assume \eqref{assumption_pe} and \eqref{thermo-relation}. Suppose that
\[
(v_0-\bar v,u_0,\theta_0-\bar\theta)
\in H^1(\mathbb{R}),
\]
and that there exist constants $c_0,C_0>0$ such that
\begin{equation}\label{initial-positive}
c_0\le v_0(x),\theta_0(x)\le C_0,
\qquad x\in\mathbb{R}.
\end{equation}
Then there exists a time $T>0$ such that the Cauchy problem
\eqref{NS}--\eqref{far-field} admits a unique strong solution on $[0,T]$ satisfying
\begin{align}
(v-\bar v,u,\theta-\bar\theta)
&\in C([0,T];H^1(\mathbb{R})), \label{local-regularity-1}\\
(u_x,\theta_x)
&\in L^2(0,T;H^1(\mathbb{R})), \label{local-regularity-2}\\
 v_t\in L^2(0,T;H^1(\mathbb R)),\qquad &
 (u_t,\theta_t)\in L^2(0,T;L^2(\mathbb R)), \label{local-regularity-3}\\
\frac{c_0}{2}\le v(t,x),\theta(t,x)
&\le 2C_0,
\qquad (t,x)\in[0,T]\times\mathbb{R}. \label{local-positive}
\end{align}
The existence time may depend on the $H^1$ size of the initial perturbation, on $c_0,C_0$, and on suitable $C^2$ bounds of the constitutive functions on the corresponding compact thermodynamic region. No sign condition on $p_v$ or $p_\theta$ is required for this local result.
\end{theorem}

\begin{remark}\label{rem:local}
The point of Theorem \ref{thm:local-large} is that local strong solvability is governed by regularity and local parabolic non-degeneracy, represented here by $e_\theta>0$, rather than by global mechanical or thermodynamic stability of the pressure. The condition $p_v(\bar v,\bar\theta)<0$ enters only when we seek a coercive global-in-time energy estimate near the equilibrium.
\end{remark}

We next describe the basic energy structure used in the global argument. Define the relative thermodynamic potential
\begin{equation}\label{relative-H}
\mathcal H(v,\theta\mid\bar v,\bar\theta)
:=
e(v,\theta)-e(\bar v,\bar\theta)
+\bar p(v-\bar v)
-\bar\theta
\bigl(s(v,\theta)-s(\bar v,\bar\theta)\bigr).
\end{equation}
This quantity is the thermodynamic part of the relative total energy with respect to the constant equilibrium $(\bar v,0,\bar\theta)$.

The following proposition provides the basic relative-energy structure used in the global argument.

\begin{proposition}[Relative-energy identity and local coercivity]\label{prop:relative-energy}
Assume \eqref{assumption_pe} and \eqref{thermo-relation}, and let $s(v,\theta)$ be determined by \eqref{Gibbs}. Then every sufficiently regular solution of \eqref{NS}--\eqref{far-field} satisfies
\begin{multline}\label{relative-energy-identity}
\int_{\mathbb{R}}
\left[
\frac12u^2
+\mathcal H(v,\theta\mid\bar v,\bar\theta)
\right](t,x)\,\mathrm{d}x
+\bar\theta
\int_0^t\int_{\mathbb{R}}
\left(
\frac{\mu u_x^2}{v\theta}
+
\frac{\kappa\theta_x^2}{v\theta^2}
\right)
\,\mathrm{d}x\,\mathrm{d}\tau\\
=
\int_{\mathbb{R}}
\left[
\frac12u_0^2
+\mathcal H(v_0,\theta_0\mid\bar v,\bar\theta)
\right]
\,\mathrm{d}x.
\end{multline}
Moreover, if the mechanical stability condition
\begin{equation}\label{mechanical-stability}
p_v(\bar v,\bar\theta)<0
\end{equation}
holds, then $\mathcal H$ is locally nonnegative and quadratically coercive near $(\bar v,\bar\theta)$. More precisely, there exist $\delta>0$ and constants $0<c_*<C_*$ such that
\begin{equation}\label{coercivity-relative-energy}
c_*
\left(
|v-\bar v|^2+|\theta-\bar\theta|^2
\right)
\le
\mathcal H(v,\theta\mid\bar v,\bar\theta)
\le
C_*
\left(
|v-\bar v|^2+|\theta-\bar\theta|^2
\right)
\end{equation}
whenever
\[
|v-\bar v|+|\theta-\bar\theta|\le\delta.
\]
\end{proposition}

We can now state the global result.

\begin{theorem}[Global existence for small perturbations]\label{thm:global-small}
Assume \eqref{assumption_pe}, \eqref{thermo-relation}, and the mechanical stability condition \eqref{mechanical-stability}. Suppose that the initial data satisfy the assumptions of Theorem \ref{thm:local-large}. Then there exists a sufficiently small constant $\varepsilon_0>0$ such that, if
\begin{equation}\label{ini}
\|(v_0-\bar v,u_0,\theta_0-\bar\theta)\|_{H^1}
\le\varepsilon_0,
\end{equation}
the Cauchy problem \eqref{NS}--\eqref{far-field} admits a unique global strong solution satisfying
\begin{equation}\label{global-regularity}
(v-\bar v,u,\theta-\bar\theta)
\in C([0,\infty);H^1(\mathbb{R})),
\qquad
(u_x,\theta_x)
\in L^2(0,\infty;H^1(\mathbb{R})).
\end{equation}
More precisely, writing $U=(v-\bar v,u,\theta-\bar\theta)$, one has
\begin{align}\label{global-quantitative-estimate}
 &\sup_{t\ge0}\|U(t)\|_{H^1}^2
 +\int_0^\infty\!\left(\|v_x(t)\|_{L^2}^2
       +\|(u_x,\theta_x)(t)\|_{H^1}^2\right)\,\mathrm{d}t
 \le C\|U(0)\|_{H^1}^2,\\
 &v_t\in L^2(0,\infty;H^1(\mathbb R)),\qquad
   (u_t,\theta_t)\in L^2(0,\infty;L^2(\mathbb R)).
   \label{global-time-regularity}
\end{align}
Here $C>0$ depends only on the equilibrium, the fixed physical
parameters, and the constitutive bounds, but is independent of time.
Moreover, the specific volume and the temperature remain uniformly
bounded relative to the prescribed far-field equilibrium:
\begin{equation}\label{global-positive}
\frac{\bar v}{2}\le v(t,x)\le\frac{3\bar v}{2},
\qquad
\frac{\bar\theta}{2}\le\theta(t,x)\le\frac{3\bar\theta}{2}
\end{equation}
for all $(t,x)\in[0,\infty)\times\mathbb{R}$.
\end{theorem}

\begin{theorem}[Long-time stability near a constant equilibrium]\label{thm:longtime}
Under the assumptions of Theorem \ref{thm:global-small}, let
$(v,u,\theta)$ be the global strong solution furnished by that theorem.
Then
\begin{equation}\label{longtime-gradients}
\lim_{t\to\infty}
\left(
\|v_x(t)\|_{L^2}
+\|u_x(t)\|_{L^2}
+\|\theta_x(t)\|_{L^2}
\right)=0,
\end{equation}
and
\begin{equation}\label{longtime-uniform}
\lim_{t\to\infty}
\bigl(
\|v(t)-\bar v\|_{L^\infty}
+\|u(t)\|_{L^\infty}
+\|\theta(t)-\bar\theta\|_{L^\infty}
\bigr)=0.
\end{equation}
In fact, for every $2<p\le\infty$,
\begin{equation}\label{longtime-Lp}
\lim_{t\to\infty}
\|(v-\bar v,u,\theta-\bar\theta)(t)\|_{L^p}=0.
\end{equation}
\end{theorem}

\begin{theorem}[Full $H^1$ stabilization for locally affine
pressure laws in entropy coordinates]
\label{thm:entropy-affine-H1}
Assume the hypotheses of Theorem~\ref{thm:global-small},
and let $s(v,\theta)$ satisfy the Gibbs relation \eqref{Gibbs}.
Suppose that, in some neighborhood $\mathcal N$ of
$(\bar v,\bar\theta)$, the pressure is affine in the
thermodynamic coordinates $(v,s)$, in the sense that
\begin{equation}\label{entropy-affine-pressure}
p(v,\theta)=\bar p-a(v-\bar v)
+b\bigl(s(v,\theta)-s(\bar v,\bar\theta)\bigr),
\qquad a>0,\quad b\in\mathbb R.
\end{equation}
For sufficiently small initial $H^1$ perturbations,
so that the resulting solution remains in $\mathcal N$,
it satisfies
\begin{equation}\label{entropy-affine-H1-decay}
\lim_{t\to\infty}
\|(v-\bar v,u,\theta-\bar\theta)(t)\|_{H^1}=0.
\end{equation}
\end{theorem}

\begin{corollary}[A temperature-affine special case]
\label{cor:structural-H1}
Under the assumptions of Theorem~\ref{thm:global-small}, suppose that
in a neighborhood of the equilibrium
\begin{equation}\label{affine-thermal-structure}
 p(v,\theta)=\bar p-\alpha(v-\bar v)+\beta(\theta-\bar\theta),
 \quad \alpha>0,\quad
 e_\theta(v,\theta)=c_0\theta,\quad c_0>0,
\end{equation}
where $\beta=p_\theta(v,\theta)$ is a constant and
\[
\gamma=\frac{\beta}{c_0}.
\]
Then sufficiently small global strong solutions converge in
$H^1(\mathbb R)$ to the equilibrium.
\end{corollary}

\begin{proof}
If $\beta=0$, the pressure depends locally only on $v$, so
\eqref{entropy-affine-pressure} holds with $a=\alpha$ and $b=0$.

Otherwise $\beta\neq0$ and Gibbs' relation gives
\[
s_v=\beta,
\qquad
s_\theta=\frac{e_\theta}{\theta}=c_0
=\frac{\beta}{\gamma}.
\]
Consequently,
\[
s-\bar s
=
\beta(v-\bar v)+\frac{\beta}{\gamma}
(\theta-\bar\theta),
\]
and substitution yields
\[
 p=\bar p-(\alpha+\beta\gamma)(v-\bar v)
       +\gamma(s-\bar s).
\]
Here
\[
\alpha+\beta\gamma
=
\alpha+\frac{\beta^2}{c_0}>0.
\]
Therefore the entropy-affine condition
\eqref{entropy-affine-pressure} is satisfied with
\[
a=\alpha+\beta\gamma,
\qquad
b=\gamma .
\]
Theorem~\ref{thm:entropy-affine-H1} applies.
\end{proof}


\begin{remark}[A distinction concerning $H^1$ convergence]\label{rem:longtime-H1}
The general-pressure conclusions of Theorem~\ref{thm:longtime}
alone do not assert
$\|(v-\bar v,u,\theta-\bar\theta)(t)\|_{L^2}\to0$.
On the whole line, decay of the spatial gradients and a uniform $L^2$
bound do not by themselves imply decay of the $L^2$ norm; additional
low-frequency control would be needed for an $H^1$ convergence result.
For example, if a particular solution is also known to satisfy
$\sup_{t\ge0}\|(v-\bar v,u,\theta-\bar\theta)(t)\|_{L^1}<\infty$,
then the one-dimensional interpolation inequality
$\|f\|_{L^2}\le C\|f\|_{L^1}^{2/3}\|f_x\|_{L^2}^{1/3}$
combines with \eqref{longtime-gradients} to give convergence in $H^1$.
The required uniform $L^1$ bound is \emph{not} established by
Theorem \ref{thm:global-small} and is not assumed in
Theorem \ref{thm:longtime}.

In contrast, Theorem~\ref{thm:entropy-affine-H1} establishes qualitative
$H^1$ convergence without requiring $L^1$ or negative-order
assumptions on the initial perturbation, or constant specific heat.
No quantitative decay rate is obtained. The entropy-affine condition
is imposed only near the equilibrium and is not assumed for general
thermodynamically compatible pressure laws.
\end{remark}

The role of the two thermodynamic sign conditions \eqref{specific-heat} and \eqref{mechanical-stability} is now transparent. The inequality $e_\theta>0$ guarantees the non-degeneracy of the temperature equation and contributes the positive thermal component
$e_\theta(\bar v,\bar\theta)/\bar\theta$ to the Hessian of the relative energy. The condition $p_v(\bar v,\bar\theta)<0$ is a mechanical stability condition and provides the positive specific-volume component $-p_v(\bar v,\bar\theta)$. Neither condition requires that the pressure be globally monotone, and no separated dependence of $p$ on $v$ and $\theta$ is imposed.

\paragraph{Interpretation of the conclusions.}
Theorems~\ref{thm:local-large} and \ref{thm:global-small} treat the
general Gibbs-compatible constitutive class, with mechanical stability
required only for global existence. Theorem~\ref{thm:longtime} gives
gradient decay and $L^p$ convergence for $p>2$, not unconditional
$L^2$ convergence. Theorem~\ref{thm:entropy-affine-H1} obtains full
$H^1$ convergence under the additional entropy-affine identity;
Proposition~\ref{prop:entropy-pressure-criterion} gives a separate,
conditional result for a general law. In particular, even the standard
ideal polytropic law does not belong to the entropy-affine class
(Remark~\ref{rem:ideal-not-entropy-affine}).

The proof is organized as follows. Section~\ref{sec:local} constructs
the local strong solution; Sections~\ref{sec:relative-energy} and
\ref{sec:global} establish relative-energy estimates and global
continuation. Section~\ref{sec:longtime} proves gradient and $L^p$
decay. Section~\ref{sec:lowfrequency} analyzes low-frequency decay
and gives a conditional criterion for nonlinear $H^1$ convergence;
Section~\ref{sec:entropy-frequency} develops the Duhamel and
entropy-coordinate arguments leading to the structural $H^1$
stabilization theorem. Section~\ref{sec:constitutive-models}
returns to the van der Waals motivation and verifies the exact scope
of the constitutive assumptions.

\section{Local Existence of Strong Solutions}\label{sec:local}

In this section we prove Theorem~\ref{thm:local-large}. Throughout the
proof, $C>0$ denotes a generic constant which may depend on
$c_0,C_0,\mu,\kappa$ and on the constitutive functions on a fixed compact
subset of $\mathbb R_+^2$, but not on the iteration index or on the local
existence time. We write $A\lesssim B$ when $A\le CB$ for such a
constant $C$. The notation $C_\eta$ allows additional dependence on
$\eta>0$, and $P(K)$ denotes a generic positive polynomial in $K$
whose coefficients depend only on the fixed physical parameters and
the constitutive bounds on the relevant compact thermodynamic region.

Recall that
\[
 c(v,\theta):=e_\theta(v,\theta)>0.
\]
The system is equivalently written as
\begin{subequations}\label{NS1}
\begin{align}
 v_t-u_x&=0, \label{NS1a}\\
 u_t+p(v,\theta)_x&=\left(\frac{\mu u_x}{v}\right)_x, \label{NS1b}\\
 c(v,\theta)\theta_t+\theta p_\theta(v,\theta)u_x
 &=\left(\frac{\kappa\theta_x}{v}\right)_x+\frac{\mu u_x^2}{v}. \label{NS1c}
\end{align}
\end{subequations}

Let
\[
 \mathcal Q:=\left[\frac{c_0}{4},4C_0\right]
 \times\left[\frac{c_0}{4},4C_0\right].
\]
Since $c=e_\theta$ is continuous and strictly positive, there are constants
$0<c_*<c^*$ such that
\begin{equation}\label{c-bounds}
 0<c_*\le c(v,\theta)\le c^*,\qquad (v,\theta)\in\mathcal Q.
\end{equation}
Moreover, because $p,e\in C^2$, all first derivatives of $c$ and all first
and second derivatives of $p$ which occur below are bounded on $\mathcal Q$.

For $T>0$, define
\begin{align*}
 \|(v,u,\theta)\|_{\mathcal X_T}:={}&
 \|(v-\bar v,u,\theta-\bar\theta)\|_{L^\infty(0,T;H^1)}\\
 &+\|(u_x,\theta_x)\|_{L^2(0,T;H^1)}
 +\|v_t\|_{L^2(0,T;H^1)}
 +\|(u_t,\theta_t)\|_{L^2(0,T;L^2)}.
\end{align*}

\subsection{Approximate scheme and uniform estimates}

A minor point is important here. Since the initial data are assumed only in
$H^1$, taking the zeroth iterate to be time independent would not place the
parabolic components in the regularity class used below. We therefore
regularize only $u$ and $\theta$ by the heat semigroup and set
\begin{equation}\label{zeroth-iterate}
 v^{(0)}(t)=v_0,\qquad
 u^{(0)}(t)=e^{t\partial_x^2}u_0,\qquad
 \theta^{(0)}(t)=\bar\theta+e^{t\partial_x^2}(\theta_0-\bar\theta).
\end{equation}
The heat semigroup preserves the pointwise bounds on $\theta_0$ and satisfies
\begin{equation}\label{heat-seed-bound}
\begin{aligned}
 &\|u^{(0)}\|_{L^\infty_tH^1_x}
 +\|u_x^{(0)}\|_{L^2_tH^1_x}
 +\|u_t^{(0)}\|_{L^2_tL^2_x}\\
 &\qquad
 +\|\theta^{(0)}-\bar\theta\|_{L^\infty_tH^1_x}
 +\|\theta_x^{(0)}\|_{L^2_tH^1_x}
 +\|\theta_t^{(0)}\|_{L^2_tL^2_x}
 \le C\|(u_0,\theta_0-\bar\theta)\|_{H^1}.
\end{aligned}
\end{equation}
Indeed, for $g\in H^1(\mathbb R)$, Plancherel's theorem gives
\[
 \int_0^T\|\partial_x^2e^{t\partial_x^2}g\|_2^2\,\mathrm{d}t
 =\frac12\int_{\mathbb R}\xi^2(1-e^{-2T\xi^2})
   |\widehat g(\xi)|^2\,\mathrm{d}\xi
 \le\frac12\|g_x\|_2^2.
\]
Since $\partial_te^{t\partial_x^2}g=\partial_x^2e^{t\partial_x^2}g$,
this proves the time-derivative and $L^2_tH^2_x$ parts of
\eqref{heat-seed-bound} from only $H^1$ initial data.

Assume that $(v^{(n-1)},u^{(n-1)},\theta^{(n-1)})$ is known. Define
\begin{equation}\label{iter-v}
 v^{(n)}(t,x)
 =v_0(x)+\int_0^t u_x^{(n-1)}(\tau,x)\,\mathrm{d}\tau,
\end{equation}
and let $u^{(n)}$ and $\theta^{(n)}$ solve
\begin{equation}\label{iter-u}
\left\{
\begin{aligned}
 &u_t^{(n)}-\left(a^{(n-1)}u_x^{(n)}\right)_x=-p_x^{(n-1)},\\
 &u^{(n)}(0)=u_0,
\end{aligned}
\right.
\end{equation}
and
\begin{equation}\label{iter-theta}
\left\{
\begin{aligned}
 &c^{(n-1)}\theta_t^{(n)}
 -\left(b^{(n-1)}\theta_x^{(n)}\right)_x=F^{(n-1)},\\
 &\theta^{(n)}(0)=\theta_0,
\end{aligned}
\right.
\end{equation}
where
\begin{align}\label{iter-coefficients}
 a^{(n-1)}&=\frac{\mu}{v^{(n-1)}},\qquad
 b^{(n-1)}=\frac{\kappa}{v^{(n-1)}},\qquad
 c^{(n-1)}=c(v^{(n-1)},\theta^{(n-1)}),\nonumber\\
 p^{(n-1)}&=p(v^{(n-1)},\theta^{(n-1)}),\nonumber\\
 F^{(n-1)}&=
 \frac{\mu}{v^{(n-1)}}(u_x^{(n-1)})^2
 -\theta^{(n-1)}
 p_\theta(v^{(n-1)},\theta^{(n-1)})u_x^{(n-1)}.
\end{align}
As long as $(v^{(n-1)},\theta^{(n-1)})\in\mathcal Q$,
equations \eqref{iter-u}--\eqref{iter-theta} are
uniformly parabolic linear equations. The scalar
linear step, stated precisely in
Lemma~\ref{lem:local-linear} below, yields a unique
strong solution once the inductive coefficient and
forcing bounds are inserted. For the temperature
equation, the lemma is applied to
$\theta^{(n)}-\bar\theta$, whose initial value belongs
to $H^1(\mathbb R)$.
In particular, $p_x^{(n-1)}\in L^2_tL^2_x$, and the
temperature source belongs to $L^2_tL^2_x$ by
\eqref{f-L2-time}. These forcing bounds follow
solely from the induction hypothesis.
The estimates are first justified for smooth
approximations and are uniform in the
approximation index.

\begin{lemma}[Scalar linear parabolic step]\label{lem:local-linear}
Let $T<\infty$ and assume $\omega,\sigma$ are measurable,
$0<\omega_*\le\omega\le\omega^*$,
$0<\sigma_*\le\sigma\le\sigma^*$, and
$\omega_x,\sigma_x\in L^\infty(0,T;L^2(\mathbb R))$.
If $z_0\in H^1(\mathbb R)$ and $f\in L^2(0,T;L^2)$, then
\[
 \omega z_t-(\sigma z_x)_x=f,\qquad z(0)=z_0,
\]
admits a unique solution satisfying
$z\in C([0,T];H^1)\cap L^2(0,T;H^2)$ and
$z_t\in L^2(0,T;L^2)$. The norms depend on the stated coefficient
bounds, the $L^\infty_tL^2_x$ norms of their spatial derivatives,
$T$, $\|z_0\|_{H^1}$ and $\|f\|_{L^2_tL^2_x}$.
\end{lemma}
\begin{proof}
Mollify coefficients and data, retaining the positive lower and
upper bounds by convolution (using constant extensions at the time
endpoints), and first solve the smooth uniformly parabolic problems.
Write $q=\sigma/\omega$, $r=\sigma_x/\omega$, and $g=f/\omega$;
then $z_t-qz_{xx}=rz_x+g$. Testing by $-z_{xx}$ gives
\[
 \frac12\frac{\mathrm{d}}{\mathrm{d}t}\|z_x\|_2^2+q_*\|z_{xx}\|_2^2
 \le \|r\|_2\|z_x\|_\infty\|z_{xx}\|_2
       +\|g\|_2\|z_{xx}\|_2.
\]
Since
$\|z_x\|_\infty\le\sqrt2\|z_x\|_2^{1/2}
\|z_{xx}\|_2^{1/2}$, Young's inequality with exponents
$4/3$ and $4$ bounds the right-hand side by
\[
 \frac{q_*}{2}\|z_{xx}\|_2^2+
 C\|r\|_2^4\|z_x\|_2^2+C\|g\|_2^2.
\]
The factor $\|r(t)\|_2^4$ is integrable on every finite time
interval because $r\in L^\infty_tL^2_x$. Gronwall therefore bounds
$z_x$ in $L^\infty_tL^2_x$ and $z_{xx}$ in $L^2_tL^2_x$.
To check the time derivative without differentiating $\omega$ or
$\sigma$ in $t$, write
\[
 \|z_t\|_2\le \|q\|_\infty\|z_{xx}\|_2
 +\|r\|_2\|z_x\|_\infty+\|g\|_2.
\]
Moreover, $\|r z_x\|_2^2\le
C\|r\|_2^2\|z_x\|_2\|z_{xx}\|_2$ is integrable in time
by the just established estimates. Hence $z_t\in L^2_tL^2_x$.
Consequently
\[
 \|z(t)\|_2\le\|z_0\|_2+T^{1/2}\|z_t\|_{L^2_tL^2_x},
\]
which supplies the missing $L^\infty_tL^2_x$ bound. All these bounds are uniform under mollification. To make the
limit passage explicit, choose bounded smooth approximations
$\omega^j,\sigma^j$ that converge almost everywhere and in
$L^2_{\mathrm{loc}}((0,T)\times\mathbb R)$ to the original
coefficients, with the same positivity bounds and uniformly bounded
$L^\infty_tL^2_x$ spatial derivatives. The solutions are bounded
in $L^2_tH^2_x\cap H^1_tL^2_x$ and in $L^\infty_tH^1_x$.
After taking a subsequence, the time derivatives and second spatial
derivatives converge weakly in $L^2$. For each compactly supported
smooth test function $\varphi$, boundedness and almost-everywhere
convergence of the coefficients give
$\omega^j\varphi\to\omega\varphi$ and
$\sigma^j\varphi_x\to\sigma\varphi_x$ strongly in
$L^2_{t,x}$ by dominated convergence. We may therefore pass to
the distributional formulation
$\int (\omega z_t\varphi+\sigma z_x\varphi_x-f\varphi)=0$;
this also identifies the weak $L^2$ limit of $z_t$.
The Hilbert-space trace and interpolation theorem, applied to
$z\in L^2(0,T;H^2)$ and $z_t\in L^2(0,T;L^2)$, gives
$z\in C([0,T];H^1)$ and identifies its initial trace as $z_0$
by weak continuity of the trace at $t=0$. For completeness, uniqueness does not require a time derivative of
$\omega$.  If $h$ is the difference of two solutions, then
$h_t-qh_{xx}=rh_x$, with $q=\sigma/\omega$ and
$r=\sigma_x/\omega$.  Multiplying by $h$ and integrating in $x$
gives
\[
 \frac12\frac{\mathrm{d}}{\mathrm{d}t}\|h\|_2^2
 +\int_{\mathbb R}q|h_x|^2\,\mathrm{d}x
 =\int_{\mathbb R}(r-q_x)h_xh\,\mathrm{d}x.
\]
In fact $r-q_x=\sigma\omega_x/\omega^2$, so
\[
 \|r-q_x\|_2
 \le C\|\omega_x\|_2=:M(t),\qquad M\in L^\infty(0,T).
\]
The one-dimensional interpolation inequality and Young's inequality
imply
\[
 \frac{\mathrm{d}}{\mathrm{d}t}\|h\|_2^2
 +q_*\|h_x\|_2^2
 \le C M(t)^4\|h\|_2^2.
\]
Gronwall's inequality and $h(0)=0$ yield $h=0$.
The coefficient hypotheses are spatial: no time derivative of either
coefficient is assumed.
\end{proof}

\begin{lemma}\label{lem:uniform-local}
There exist $K\ge1$ and $T_0>0$ such that, for every $0<T\le T_0$ and every
$n\ge0$,
\begin{equation}\label{uniform-local}
 \|(v^{(n)},u^{(n)},\theta^{(n)})\|_{\mathcal X_T}\le K,
 \qquad
 \frac{c_0}{2}\le v^{(n)},\theta^{(n)}\le2C_0.
\end{equation}
\end{lemma}

\begin{proof}
Choose $K$ larger than a fixed multiple of
\[
1+\|(v_0-\bar v,u_0,\theta_0-\bar\theta)\|_{H^1}.
\]
We argue by induction. The estimate for the zeroth iterate follows from
\eqref{heat-seed-bound}. Assume that \eqref{uniform-local} holds at level
$n-1$.

\medskip
\noindent\textit{Step 1: Estimate for $v^{(n)}$.}
From \eqref{iter-v},
\begin{equation}\label{local-v}
 \|v^{(n)}-\bar v\|_{L^\infty_tH^1_x}
 \le \|v_0-\bar v\|_{H^1}
 +T^{1/2}\|u_x^{(n-1)}\|_{L^2_tH^1_x},
\end{equation}
and
\[
v_t^{(n)}=u_x^{(n-1)}.
\]

\medskip
\noindent\textit{Step 2: Estimates for $u^{(n)}$.}
Testing \eqref{iter-u} by $u^{(n)}$ and integrating over $\mathbb R$, we obtain
\[
\frac{\mathrm{d}}{\mathrm{d}t}\|u^{(n)}(t)\|^2
+\|u_x^{(n)}(t)\|^2
\lesssim
\|u^{(n)}(t)\|^2
+\|(v_x^{(n-1)},\theta_x^{(n-1)})(t)\|^2.
\]
Hence, by the induction hypothesis and Gronwall's inequality, for $T_0\le1$,
\begin{equation}\label{local-u0}
 \|u^{(n)}\|_{L^\infty_tL^2_x}^2
 +\|u_x^{(n)}\|_{L^2_tL^2_x}^2
 \lesssim \|u_0\|^2+TK^2.
\end{equation}

Next, testing \eqref{iter-u} by $-u_{xx}^{(n)}$ gives
\begin{equation}\label{u11}
 \frac{\mathrm{d}}{\mathrm{d}t}\|u_x^{(n)}(t)\|^2
 +\|u_{xx}^{(n)}(t)\|^2
 \lesssim
 \int_{\mathbb R}|p_x^{(n-1)}u_{xx}^{(n)}|\,\mathrm{d}x
 +\int_{\mathbb R}|a_x^{(n-1)}u_x^{(n)}u_{xx}^{(n)}|\,\mathrm{d}x .
\end{equation}
Since the iterates remain in the fixed compact thermodynamic region,
\begin{equation}\label{px-L2}
 \|p_x^{(n-1)}\|^2
 \lesssim
 \|(v_x^{(n-1)},\theta_x^{(n-1)})\|^2
 \lesssim K^2.
\end{equation}
Moreover, $|a_x^{(n-1)}|\lesssim |v_x^{(n-1)}|$, and the
one-dimensional inequality
\begin{equation}\label{eg2}
 \|g\|_{L^\infty}
 \le \sqrt2\,\|g\|^{1/2}\|g_x\|^{1/2},
 \qquad g\in H^1(\mathbb R),
\end{equation}
implies, for every $\eta>0$,
\begin{align}
 \int_{\mathbb R}|a_x^{(n-1)}u_x^{(n)}u_{xx}^{(n)}|\,\mathrm{d}x
 &\lesssim
 \|v_x^{(n-1)}\|
 \|u_x^{(n)}\|^{1/2}
 \|u_{xx}^{(n)}\|^{3/2}\nonumber\\
 &\le
 \eta\|u_{xx}^{(n)}\|^2
 +C_\eta K^4\|u_x^{(n)}\|^2.
\label{eg1}
\end{align}
Using Young's inequality for the pressure term in \eqref{u11}, choosing
$\eta$ sufficiently small, and then applying Gronwall's inequality together
with \eqref{local-u0}, we obtain, provided $TK^4$ is sufficiently small,
\begin{equation}\label{local-u1}
 \|u^{(n)}\|_{L^\infty_tH^1_x}^2
 +\|u_x^{(n)}\|_{L^2_tH^1_x}^2
 \lesssim
 \|u_0\|_{H^1}^2+TK^2.
\end{equation}

To estimate the time derivative, we use \eqref{iter-u} directly:
\[
u_t^{(n)}
=a^{(n-1)}u_{xx}^{(n)}
+a_x^{(n-1)}u_x^{(n)}
-p_x^{(n-1)}.
\]
Therefore,
\[
\|u_t^{(n)}\|^2
\lesssim
\|u_{xx}^{(n)}\|^2
+\|p_x^{(n-1)}\|^2
+\|v_x^{(n-1)}u_x^{(n)}\|^2.
\]
By \eqref{eg2},
\[
\int_0^T\|v_x^{(n-1)}u_x^{(n)}\|^2\,\mathrm{d}t
\lesssim
K^2T^{1/2}
\|u_x^{(n)}\|_{L^\infty_tL^2_x}
\|u_{xx}^{(n)}\|_{L^2_tL^2_x}
\le
T^{1/2}P(K).
\]
Combining this with \eqref{px-L2} and \eqref{local-u1}, and reducing $T_0$
if necessary, yields
\begin{equation}\label{local-u-est}
 \|u^{(n)}\|_{L^\infty_tH^1_x}^2
 +\|u_x^{(n)}\|_{L^2_tH^1_x}^2
 +\|u_t^{(n)}\|_{L^2_tL^2_x}^2
 \le C\|u_0\|_{H^1}^2+T^{1/2}P(K),
\end{equation}
where $P$ is a polynomial independent of $n$.

\medskip
\noindent\textit{Step 3: Estimates for $\theta^{(n)}$.}
For the temperature equation it is convenient to avoid differentiating the
coefficient $c=e_\theta$. Dividing \eqref{iter-theta} by $c^{(n-1)}$, we write
\begin{equation}\label{theta-iter-normalized}
 \theta_t^{(n)}
 -q^{(n-1)}\theta_{xx}^{(n)}
 =r^{(n-1)}\theta_x^{(n)}+f^{(n-1)},
\end{equation}
where
\begin{equation}\label{qrf-local}
 q^{(n-1)}=\frac{b^{(n-1)}}{c^{(n-1)}},
 \qquad
 r^{(n-1)}=\frac{b_x^{(n-1)}}{c^{(n-1)}},
 \qquad
 f^{(n-1)}=\frac{F^{(n-1)}}{c^{(n-1)}}.
\end{equation}
By \eqref{c-bounds}, $q^{(n-1)}\ge q_*>0$ uniformly in $n$.

Set $w^{(n)}=\theta^{(n)}-\bar\theta$. Testing
\eqref{theta-iter-normalized} by $w^{(n)}$ gives
\begin{align}
 \frac12\frac{\mathrm{d}}{\mathrm{d}t}\|w^{(n)}\|^2
 +\int_{\mathbb R}q^{(n-1)}|\theta_x^{(n)}|^2\,\mathrm{d}x
 ={}&
 \int_{\mathbb R}
 \bigl(r^{(n-1)}-(q^{(n-1)})_x\bigr)
 \theta_x^{(n)}w^{(n)}\,\mathrm{d}x\nonumber\\
 &+\int_{\mathbb R}f^{(n-1)}w^{(n)}\,\mathrm{d}x.
\label{theta-L2-energy}
\end{align}
Since
\[
 r^{(n-1)}-(q^{(n-1)})_x
 =
 \frac{b^{(n-1)}(c^{(n-1)})_x}{(c^{(n-1)})^2},
\]
and $c=e_\theta\in C^1(\mathcal Q)$,
\[
 |(c^{(n-1)})_x|
 \le
 C\bigl(|v_x^{(n-1)}|+|\theta_x^{(n-1)}|\bigr).
\]
Thus, by \eqref{eg2} and Young's inequality,
\begin{align}
 &\left|
 \int_{\mathbb R}
 \bigl(r^{(n-1)}-(q^{(n-1)})_x\bigr)
 \theta_x^{(n)}w^{(n)}\,\mathrm{d}x
 \right|\nonumber\\
 &\qquad\le
 \eta\|\theta_x^{(n)}\|^2
 +C_\eta K^4\|w^{(n)}\|^2.
\label{theta-coefficient-term}
\end{align}

The source satisfies
\begin{equation}\label{f-pointwise}
 |f^{(n-1)}|
 \le
 C\bigl(|u_x^{(n-1)}|^2+|u_x^{(n-1)}|\bigr).
\end{equation}
Using the one-dimensional Gagliardo--Nirenberg inequality,
\[
 \|g\|_{L^4}^4\lesssim \|g\|^3\|g_x\|,
\]
and the induction hypothesis, the exponent in the source bound is
accounted for explicitly by
\[
 \int_0^T\|u_x^{(n-1)}\|_{L^4}^4\,\mathrm{d}t
 \le C\|u_x^{(n-1)}\|_{L^\infty_tL^2_x}^3
       T^{1/2}\|u_{xx}^{(n-1)}\|_{L^2_tL^2_x}
 \le CT^{1/2}K^4.
\]
Also $\int_0^T\|u_x^{(n-1)}\|_2^2\,\mathrm{d}t\le TK^2$.
Thus we obtain
\begin{align}
 \int_0^T\|f^{(n-1)}\|^2\,\mathrm{d}t
 &\lesssim
 \int_0^T
 \left(
 \|u_x^{(n-1)}\|_{L^4}^4
 +\|u_x^{(n-1)}\|^2
 \right)\mathrm{d}t\nonumber\\
 &\le T^{1/2}P(K).
\label{f-L2-time}
\end{align}
Consequently,
\[
\left|
\int_{\mathbb R}f^{(n-1)}w^{(n)}\,\mathrm{d}x
\right|
\le
\frac12\|w^{(n)}\|^2
+C\|f^{(n-1)}\|^2.
\]
Substituting these bounds into \eqref{theta-L2-energy}, choosing $\eta$ small,
and applying Gronwall's inequality give
\begin{equation}\label{theta-L2-est}
 \|\theta^{(n)}-\bar\theta\|_{L^\infty_tL^2_x}^2
 +\|\theta_x^{(n)}\|_{L^2_tL^2_x}^2
 \le
 C\|\theta_0-\bar\theta\|^2
 +T^{1/2}P(K),
\end{equation}
provided $T(1+K^4)$ is sufficiently small.

We next test \eqref{theta-iter-normalized} by
$-\theta_{xx}^{(n)}$. This yields
\begin{align}
 \frac12\frac{\mathrm{d}}{\mathrm{d}t}\|\theta_x^{(n)}\|^2
 +\int_{\mathbb R}q^{(n-1)}|\theta_{xx}^{(n)}|^2\,\mathrm{d}x
 ={}&
 -\int_{\mathbb R}
 r^{(n-1)}\theta_x^{(n)}\theta_{xx}^{(n)}\,\mathrm{d}x\nonumber\\
 &-\int_{\mathbb R}
 f^{(n-1)}\theta_{xx}^{(n)}\,\mathrm{d}x.
\label{theta-H1-energy}
\end{align}
Since $|r^{(n-1)}|\lesssim |v_x^{(n-1)}|$, the first term satisfies
\[
\left|
\int_{\mathbb R}
r^{(n-1)}\theta_x^{(n)}\theta_{xx}^{(n)}\,\mathrm{d}x
\right|
\le
\eta\|\theta_{xx}^{(n)}\|^2
+C_\eta K^4\|\theta_x^{(n)}\|^2.
\]
For the source term, \eqref{f-pointwise} gives
\[
\left|
\int_{\mathbb R}
f^{(n-1)}\theta_{xx}^{(n)}\,\mathrm{d}x
\right|
\le
\eta\|\theta_{xx}^{(n)}\|^2
+C_\eta\left(
\|u_x^{(n-1)}\|_{L^4}^4
+\|u_x^{(n-1)}\|^2
\right).
\]
After integration in time, \eqref{f-L2-time} and Gronwall's inequality imply
\begin{equation}\label{theta-H1-est}
 \|\theta_x^{(n)}\|_{L^\infty_tL^2_x}^2
 +\|\theta_{xx}^{(n)}\|_{L^2_tL^2_x}^2
 \le
 C\|\theta_{0x}\|^2+T^{1/2}P(K).
\end{equation}

Finally, from \eqref{theta-iter-normalized},
\[
\|\theta_t^{(n)}\|^2
\lesssim
\|\theta_{xx}^{(n)}\|^2
+\|v_x^{(n-1)}\theta_x^{(n)}\|^2
+\|f^{(n-1)}\|^2.
\]
Using \eqref{eg2}, \eqref{theta-H1-est}, and the induction hypothesis,
\[
\int_0^T
\|v_x^{(n-1)}\theta_x^{(n)}\|^2\,\mathrm{d}t
\le
T^{1/2}P(K).
\]
Combining the preceding estimates, we conclude that
\begin{equation}\label{local-theta-est}
 \|\theta^{(n)}-\bar\theta\|_{L^\infty_tH^1_x}^2
 +\|\theta_x^{(n)}\|_{L^2_tH^1_x}^2
 +\|\theta_t^{(n)}\|_{L^2_tL^2_x}^2
 \le
 C\|\theta_0-\bar\theta\|_{H^1}^2
 +T^{1/2}P(K).
\end{equation}

\medskip
\noindent\textit{Step 4: Preservation of the pointwise bounds.}
From \eqref{iter-v} and the Sobolev embedding,
\[
 \|v^{(n)}(t)-v_0\|_{L^\infty}
 \le
 T^{1/2}\|u_x^{(n-1)}\|_{L^2_tH^1_x}.
\]
Hence $c_0/2\le v^{(n)}\le2C_0$ for sufficiently small $T$.

For the temperature, we justify the comparison estimate in the
regularity class used here. Abbreviate
$c=c^{(n-1)}$, $b=b^{(n-1)}$, and $F=F^{(n-1)}$, so that
$c\theta_t^{(n)}-(b\theta_x^{(n)})_x=F$, with $c,b$ bounded
above and away from zero. Define
\[
 h(t):=\|F(t,\cdot)/c(t,\cdot)\|_{L^\infty},\qquad
 M(t):=\operatorname*{ess\,sup}_{x\in\mathbb R}\theta_0(x)
       +\int_0^t h(\tau)\,\mathrm{d}\tau.
\]
Because $M(t)\ge\operatorname*{ess\,sup}\theta_0\ge\bar\theta$,
the positive part $z=(\theta^{(n)}-M)_+$ belongs to
$H^1(\mathbb R)$; its $L^2$ norm is finite even on the whole line.
It has zero initial trace and, since
$M'(t)=h(t)$, satisfies the one-sided energy inequality
\[
 \frac12\frac{\mathrm{d}}{\mathrm{d}t}\int_{\mathbb R}c z^2\,\mathrm{d}x
 +\int_{\mathbb R}b|z_x|^2\,\mathrm{d}x
 \le \frac12\int_{\mathbb R}c_tz^2\,\mathrm{d}x.
\]
Indeed, test the divergence-form temperature equation by $z$;
on the set $\{z>0\}$ the source $F-cM'$ is nonpositive.
Because $c_t=e_{\theta v}(v^{(n-1)},\theta^{(n-1)})v_t^{(n-1)}
+e_{\theta\theta}(v^{(n-1)},\theta^{(n-1)})\theta_t^{(n-1)}$
belongs to $L^2(0,T;L^2)$ uniformly in $n$, the one-dimensional
interpolation inequality yields, for any $\eta>0$,
\[
 \left|\int c_tz^2\,\mathrm{d}x\right|
 \le\eta\|z_x\|^2+C_\eta\|c_t\|_{L^2}^{4/3}\|z\|^2.
\]
Since $\|c_t\|_{L^2}^{4/3}\in L^1(0,T)$, Gronwall's inequality
implies $z=0$. Applying the same argument to
$(m(t)-\theta^{(n)})_+$, where
$m(t)=\operatorname*{ess\,inf}\theta_0-\int_0^t h(\tau)\,\mathrm{d}\tau$,
gives the two-sided comparison bound
\[
 \operatorname*{ess\,inf}\theta_0
 -\int_0^t\|f^{(n-1)}(\tau)\|_{L^\infty}\,\mathrm{d}\tau
 \le
 \theta^{(n)}(t,x)
 \le
 \operatorname*{ess\,sup}\theta_0
 +\int_0^t\|f^{(n-1)}(\tau)\|_{L^\infty}\,\mathrm{d}\tau.
\]
The testing can be justified first for smooth approximations and then
passed to the strong-solution limit; in particular, it does not require
pointwise boundedness of the drift coefficient in the normalized equation.
By \eqref{f-pointwise}, \eqref{eg2}, and the induction hypothesis,
\begin{align*}
 \int_0^T\|f^{(n-1)}\|_{L^\infty}\,\mathrm{d}t
 &\lesssim
 \int_0^T
 \left(
 \|u_x^{(n-1)}\|_{L^\infty}^2
 +\|u_x^{(n-1)}\|_{L^\infty}
 \right)\mathrm{d}t\\
 &\le T^{1/2}P(K).
\end{align*}
Thus, after reducing $T_0$ if necessary,
\[
 \frac{c_0}{2}\le \theta^{(n)}\le2C_0.
\]

Combining \eqref{local-v}, \eqref{local-u-est}, and
\eqref{local-theta-est}, and choosing $T_0>0$ sufficiently small in terms of
$K$ and the fixed constitutive bounds, we obtain
\[
 \|(v^{(n)},u^{(n)},\theta^{(n)})\|_{\mathcal X_T}\le K
\]
for every $0<T\le T_0$. This closes the induction and proves the lemma.
\end{proof}

\subsection{Contraction of the Picard iteration}\label{subsec:picard-contraction}

We now prove convergence of the approximate sequence constructed in Subsection~2.1.
We carry out the contraction in a lower-order metric.  This is sufficient for the
fixed-point argument and, at the same time, avoids a loss of one spatial derivative
in the differences of the variable coefficients.  The full strong-solution
regularity is recovered after passage to the limit from the uniform estimates of
Subsection~2.1.  In particular, an $H^1$ contraction of the differences is neither
needed nor natural under the present $C^2$ assumptions on the constitutive laws.

Recall that, for some $K\ge1$ and $T_0>0$, the iterates satisfy, uniformly in $j$,
\begin{equation}\label{standing-uniform-bound}
\begin{aligned}
 &\|(v^{(j)}-\bar v,u^{(j)},\theta^{(j)}-\bar\theta)\|_{L^\infty(0,T;H^1)}
 +\|(u_x^{(j)},\theta_x^{(j)})\|_{L^2(0,T;H^1)}\\
 &\qquad
 +\|v_t^{(j)}\|_{L^2(0,T;H^1)}
 +\|(u_t^{(j)},\theta_t^{(j)})\|_{L^2(0,T;L^2)}
 \le K,
 \qquad 0<T\le T_0,
\end{aligned}
\end{equation}
and
\begin{equation}\label{standing-pointwise-bound}
 \frac{c_0}{2}\le v^{(j)}(t,x),\theta^{(j)}(t,x)\le 2C_0.
\end{equation}
All constants below are independent of $j$ and $T\le T_0$.

For $n\ge0$, set
\[
 \widetilde v^{(n)}=v^{(n+1)}-v^{(n)},\qquad
 \widetilde u^{(n)}=u^{(n+1)}-u^{(n)},\qquad
 \widetilde\theta^{(n)}=\theta^{(n+1)}-\theta^{(n)}.
\]
We use the same convention for the coefficients.  Thus, for example,
$\widetilde a^{(n-1)}=a^{(n)}-a^{(n-1)}$ and
$\widetilde F^{(n-1)}=F^{(n)}-F^{(n-1)}$.
Subtracting the equations at levels $n+1$ and $n$ gives, for $n\ge1$,
\begin{equation}\label{difference-v}
 \widetilde v^{(n)}(t)
 =\int_0^t\widetilde u_x^{(n-1)}(\tau)\,\mathrm{d}\tau,
 \qquad \widetilde v^{(n)}(0)=0,
\end{equation}
\begin{equation}\label{difference-u}
 \left\{
 \begin{aligned}
 &\widetilde u_t^{(n)}
 -\bigl(a^{(n)}\widetilde u_x^{(n)}\bigr)_x
 =-\widetilde p_x^{(n-1)}
 +\bigl(\widetilde a^{(n-1)}u_x^{(n)}\bigr)_x,\\
 &\widetilde u^{(n)}(0)=0,
 \end{aligned}
 \right.
\end{equation}
For the temperature equation, it is useful to check the indices directly.  At
levels $n+1$ and $n$ we have
\[
\begin{aligned}
 c^{(n)}\theta_t^{(n+1)}
 -\bigl(b^{(n)}\theta_x^{(n+1)}\bigr)_x&=F^{(n)},\\
 c^{(n-1)}\theta_t^{(n)}
 -\bigl(b^{(n-1)}\theta_x^{(n)}\bigr)_x&=F^{(n-1)}.
\end{aligned}
\]
Since $\widetilde\theta^{(n)}=\theta^{(n+1)}-\theta^{(n)}$,
\[
\begin{aligned}
 c^{(n)}\theta_t^{(n+1)}-c^{(n-1)}\theta_t^{(n)}
 &=c^{(n)}\widetilde\theta_t^{(n)}
   +\widetilde c^{(n-1)}\theta_t^{(n)},\\
 b^{(n)}\theta_x^{(n+1)}-b^{(n-1)}\theta_x^{(n)}
 &=b^{(n)}\widetilde\theta_x^{(n)}
   +\widetilde b^{(n-1)}\theta_x^{(n)}.
\end{aligned}
\]
Therefore direct subtraction gives
\begin{equation}\label{difference-theta}
 \left\{
 \begin{aligned}
 &c^{(n)}\widetilde\theta_t^{(n)}
 -\bigl(b^{(n)}\widetilde\theta_x^{(n)}\bigr)_x
 ={}\widetilde F^{(n-1)}
 +\bigl(\widetilde b^{(n-1)}\theta_x^{(n)}\bigr)_x
-\widetilde c^{(n-1)}\theta_t^{(n)},\\
& \widetilde\theta^{(n)}(0)=0.
 \end{aligned}
 \right.
\end{equation}
Thus the sign and the indices in \eqref{difference-theta} follow directly from
the definition of $\widetilde\theta^{(n)}$.  We do not differentiate this equation
in $x$: doing so would generate, among other terms,
$\widetilde c_x^{(n-1)}\theta_t^{(n)}$ and
$\widetilde c^{(n-1)}\theta_{tx}^{(n)}$, which are not controlled by the present
strong-solution norm under only $C^2$ constitutive assumptions.

Define the lower-order difference energy
\begin{equation}\label{difference-energy}
 \mathfrak D_n(T):={}
 \|(\widetilde v^{(n)},\widetilde u^{(n)},
 \widetilde\theta^{(n)})\|_{L^\infty(0,T;L^2)}^2
 +\|(\widetilde u_x^{(n)},\widetilde\theta_x^{(n)})
 \|_{L^2(0,T;L^2)}^2.
\end{equation}
By the one-dimensional inequality
$\|g\|_{L^\infty}^2\le 2\|g\|\|g_x\|$ and
\eqref{standing-uniform-bound},
\begin{equation}\label{small-L2-Linf}
 \|u_x^{(j)}\|_{L^2(0,T;L^\infty)}^2
 +\|\theta_x^{(j)}\|_{L^2(0,T;L^\infty)}^2
 \le C T^{1/2}K^2,
\end{equation}
while
\begin{equation}\label{small-L1-Linf}
 \|u_x^{(j)}\|_{L^1(0,T;L^\infty)}
 +\|\theta_x^{(j)}\|_{L^1(0,T;L^\infty)}
 \le C T^{3/4}K.
\end{equation}
Furthermore, the time derivative of the \emph{iterated} heat-capacity
coefficient is available, although no such derivative was assumed in the
abstract linear parabolic lemma.  In fact, by the chain rule and the
$C^2$ constitutive hypothesis,
\[
 (c^{(j)})_t
 =e_{\theta v}(v^{(j)},\theta^{(j)})v_t^{(j)}
  +e_{\theta\theta}(v^{(j)},\theta^{(j)})\theta_t^{(j)}.
\]
Consequently, \eqref{standing-uniform-bound} gives
\begin{equation}\label{ct-uniform-difference}
 \|c_t^{(j)}\|_{L^2(0,T;L^2)}\le P(K),
 \qquad
 \int_0^T\|c_t^{(j)}(t)\|_2^{4/3}\,\mathrm{d}t
 \le T^{1/3}P(K),
\end{equation}
where $P$ denotes a polynomial whose coefficients depend only on the fixed
constitutive bounds.  The second bound follows from H\"older's inequality
with exponents $3/2$ and $3$.  This is the precise small-time
coefficient needed when testing the \emph{difference} equation by
$\widetilde\theta^{(n)}$; it imposes no extra hypothesis on the initial
data.

\begin{lemma}[Picard contraction]\label{lem:contraction-local}
There exists $T_1\in(0,T_0]$ such that
\begin{equation}\label{contraction-estimate}
 \mathfrak D_n(T_1)\le\frac12\mathfrak D_{n-1}(T_1),
 \qquad n\ge1.
\end{equation}
\end{lemma}

\begin{proof}
We estimate the three difference equations separately.

\medskip
\noindent\emph{Step 1. The specific-volume difference.}
From \eqref{difference-v} and Cauchy--Schwarz,
\begin{equation}\label{difference-v-est}
 \|\widetilde v^{(n)}\|_{L^\infty(0,T;L^2)}^2
 \le T\|\widetilde u_x^{(n-1)}\|_{L^2(0,T;L^2)}^2
 \le T\mathfrak D_{n-1}(T).
\end{equation}

\medskip
\noindent\emph{Step 2. The velocity difference.}
The uniform pointwise bounds imply $a^{(n)}\ge a_*>0$.  Multiplying
\eqref{difference-u} by $\widetilde u^{(n)}$ and integrating by parts, we obtain
\begin{equation}\label{difference-u-energy}
\begin{aligned}
 \frac12\frac{\mathrm{d}}{\mathrm{d}t}\|\widetilde u^{(n)}\|^2
 +\int_{\mathbb R}a^{(n)}|\widetilde u_x^{(n)}|^2\,\mathrm{d}x
 ={}&\int_{\mathbb R}\widetilde p^{(n-1)}
 \widetilde u_x^{(n)}\,\mathrm{d}x\\
 &-\int_{\mathbb R}\widetilde a^{(n-1)}u_x^{(n)}
 \widetilde u_x^{(n)}\,\mathrm{d}x.
\end{aligned}
\end{equation}
The mean value theorem gives
\begin{equation}\label{coefficient-differences}
\begin{aligned}
 |\widetilde p^{(n-1)}|+|\widetilde c^{(n-1)}|
 &\le C\bigl(|\widetilde v^{(n-1)}|
 +|\widetilde\theta^{(n-1)}|\bigr),\\
 |\widetilde a^{(n-1)}|+|\widetilde b^{(n-1)}|
 &\le C|\widetilde v^{(n-1)}|.
\end{aligned}
\end{equation}
Hence, for every $\eta>0$,
\begin{align}
 \int_0^t\left|\int_{\mathbb R}
 \widetilde p^{(n-1)}\widetilde u_x^{(n)}\,\mathrm{d}x\right|\mathrm{d}\tau
 &\le \eta\|\widetilde u_x^{(n)}\|_{L^2(0,t;L^2)}^2
 +C_\eta t\,\mathfrak D_{n-1}(t),
 \label{difference-pressure-term}\\
 \int_0^t\left|\int_{\mathbb R}
 \widetilde a^{(n-1)}u_x^{(n)}\widetilde u_x^{(n)}\,\mathrm{d}x\right|\mathrm{d}\tau
 &\le \eta\|\widetilde u_x^{(n)}\|_{L^2(0,t;L^2)}^2
 +C_\eta t^{1/2}P(K)\mathfrak D_{n-1}(t),
 \label{difference-a-term}
\end{align}
where \eqref{small-L2-Linf} was used in the second estimate.  Taking the
supremum over $0\le t\le T$ and choosing $\eta$ sufficiently small yields
\begin{equation}\label{difference-u-est}
 \|\widetilde u^{(n)}\|_{L^\infty(0,T;L^2)}^2
 +\|\widetilde u_x^{(n)}\|_{L^2(0,T;L^2)}^2
 \le C\bigl(T+T^{1/2}P(K)\bigr)\mathfrak D_{n-1}(T).
\end{equation}

\medskip
\noindent\emph{Step 3. The temperature difference.}
Multiplying \eqref{difference-theta} by $\widetilde\theta^{(n)}$ and integrating by
parts gives
\begin{equation}\label{difference-theta-energy}
\begin{aligned}
 &\frac12\frac{\mathrm{d}}{\mathrm{d}t}
 \int_{\mathbb R}c^{(n)}|\widetilde\theta^{(n)}|^2\,\mathrm{d}x
 +\int_{\mathbb R}b^{(n)}|\widetilde\theta_x^{(n)}|^2\,\mathrm{d}x\\
 &\quad=
 \frac12\int_{\mathbb R}c_t^{(n)}|\widetilde\theta^{(n)}|^2\,\mathrm{d}x
 +\int_{\mathbb R}\widetilde F^{(n-1)}\widetilde\theta^{(n)}\,\mathrm{d}x\\
 &\qquad
 -\int_{\mathbb R}\widetilde b^{(n-1)}\theta_x^{(n)}
 \widetilde\theta_x^{(n)}\,\mathrm{d}x
 -\int_{\mathbb R}\widetilde c^{(n-1)}\theta_t^{(n)}
 \widetilde\theta^{(n)}\,\mathrm{d}x.
\end{aligned}
\end{equation}
For the first term, H\"older's inequality and the one-dimensional
Gagliardo--Nirenberg inequality give
\[
 \begin{aligned}
 \left|\int c_t^{(n)}|\widetilde\theta^{(n)}|^2\,\mathrm{d}x\right|
 &\le \|c_t^{(n)}\|_2\|\widetilde\theta^{(n)}\|_4^2\\
 &\le C\|c_t^{(n)}\|_2
      \|\widetilde\theta^{(n)}\|_2^{3/2}
      \|\widetilde\theta_x^{(n)}\|_2^{1/2}\\
 &\le \eta\|\widetilde\theta_x^{(n)}\|_2^2
      +C_\eta\|c_t^{(n)}\|_2^{4/3}
      \|\widetilde\theta^{(n)}\|_2^2.
 \end{aligned}
\]
The final line uses Young's exponents $4$ and $4/3$.
Together with \eqref{ct-uniform-difference}, this implies
\begin{equation}\label{difference-ct-term}
\begin{aligned}
 \int_0^t\frac12\left|\int c_t^{(n)}|\widetilde\theta^{(n)}|^2\,\mathrm{d}x\right|\mathrm{d}\tau
 \le{}&\eta\|\widetilde\theta_x^{(n)}\|_{L^2(0,t;L^2)}^2\\
 &+t^{1/3}P(K)\|\widetilde\theta^{(n)}\|_{L^\infty(0,t;L^2)}^2.
\end{aligned}
\end{equation}
By \eqref{coefficient-differences} and \eqref{small-L2-Linf},
\begin{equation}\label{difference-b-term}
\begin{aligned}
 &\int_0^t\left|\int
 \widetilde b^{(n-1)}\theta_x^{(n)}\widetilde\theta_x^{(n)}\,\mathrm{d}x\right|\mathrm{d}\tau\\
 &\qquad\le \eta\|\widetilde\theta_x^{(n)}\|_{L^2(0,t;L^2)}^2
 +C_\eta t^{1/2}P(K)\mathfrak D_{n-1}(t).
\end{aligned}
\end{equation}
The term caused by the variable specific heat is handled without
spatially differentiating $\widetilde c^{(n-1)}$ or assuming
$\theta_{tx}^{(n)}\in L^2$.  For notational clarity, put
\[
 A=\bigl\|(\widetilde v^{(n-1)},
       \widetilde\theta^{(n-1)})\bigr\|_{L^\infty(0,t;L^2)},\qquad
 Z=\|\widetilde\theta^{(n)}\|_{L^\infty(0,t;L^2)},\qquad
 B=\|\widetilde\theta_x^{(n)}\|_{L^2(0,t;L^2)}.
\]
The Lipschitz continuity of $c=e_\theta$ on the fixed state rectangle
and the one-dimensional interpolation inequality yield
\begin{align*}
 &\int_0^t\left|\int_{\mathbb R}
 \widetilde c^{(n-1)}\theta_t^{(n)}
 \widetilde\theta^{(n)}\,\mathrm{d}x\right|\mathrm{d}\tau\\
 &\qquad\le CAZ^{1/2}
 \int_0^t\|\theta_t^{(n)}(\tau)\|_2
            \|\widetilde\theta_x^{(n)}(\tau)\|_2^{1/2}
            \,\mathrm{d}\tau.
\end{align*}
The integrand is well-defined for almost every time because the
strong-solution bounds give $\theta_t^{(n)}\in L^2_tL^2_x$
and $\widetilde\theta^{(n)}\in L^2_tH^1_x$.  H\"older's
inequality in time, with exponents $(4,2,4)$, gives
\[
 \int_0^t\|\theta_t^{(n)}\|_2
        \|\widetilde\theta_x^{(n)}\|_2^{1/2}\,\mathrm{d}\tau
 \le t^{1/4}\|\theta_t^{(n)}\|_{L^2(0,t;L^2)}B^{1/2}
 \le t^{1/4}P(K)B^{1/2}.
\]
Since $Z^{1/2}B^{1/2}\le (Z+B)/2$ and
$A^2\le C\mathfrak D_{n-1}(t)$, Young's inequality now
shows, for every $\eta>0$, that
\begin{equation}\label{difference-c-term}
\begin{aligned}
 &\int_0^t\left|\int_{\mathbb R}
 \widetilde c^{(n-1)}\theta_t^{(n)}
 \widetilde\theta^{(n)}\,\mathrm{d}x\right|\mathrm{d}\tau\\
 &\qquad\le
 \eta\left(
 \|\widetilde\theta^{(n)}\|_{L^\infty(0,t;L^2)}^2
 +\|\widetilde\theta_x^{(n)}\|_{L^2(0,t;L^2)}^2\right)
 +C_\eta t^{1/2}P(K)\mathfrak D_{n-1}(t).
\end{aligned}
\end{equation}
In particular, the coefficient difference is estimated in $L^2_x$,
not $L^\infty_x$: it is the \emph{current} temperature difference
whose $H^1_x$ interpolation supplies the missing factor.  This is
why the available $L^2_tL^2_x$ bound for $\theta_t^{(n)}$ is sufficient.

It remains to estimate $\widetilde F^{(n-1)}$.  From the definition of
$F^{(j)}$ we have the exact decomposition
\begin{equation}\label{difference-F-decomposition}
\begin{aligned}
 \widetilde F^{(n-1)}={}&
 \widetilde a^{(n-1)}(u_x^{(n)})^2
 +a^{(n-1)}(u_x^{(n)}+u_x^{(n-1)})\widetilde u_x^{(n-1)}\\
 &-\Bigl[\theta^{(n)}p_\theta(v^{(n)},\theta^{(n)})
 -\theta^{(n-1)}p_\theta(v^{(n-1)},\theta^{(n-1)})\Bigr]u_x^{(n)}\\
 &-\theta^{(n-1)}p_\theta(v^{(n-1)},\theta^{(n-1)})
 \widetilde u_x^{(n-1)}.
\end{aligned}
\end{equation}
By the mean value theorem,
\[
 \left|\theta^{(n)}p_\theta(v^{(n)},\theta^{(n)})
 -\theta^{(n-1)}p_\theta(v^{(n-1)},\theta^{(n-1)})\right|
 \le C\bigl(|\widetilde v^{(n-1)}|+|\widetilde\theta^{(n-1)}|\bigr).
\]
For clarity, denote the four summands in
\eqref{difference-F-decomposition} by $I_1,I_2,I_3,I_4$.
By \eqref{small-L2-Linf}--\eqref{small-L1-Linf},
\eqref{coefficient-differences}, and Young's inequality, we have
\begin{align*}
 \int_0^t\!\left|\int_{\mathbb R}I_1\widetilde\theta^{(n)}\,\mathrm{d}x\right|\mathrm{d}\tau
 &\le C \mathfrak D_{n-1}(t)^{1/2}Z\|u_x^{(n)}\|_{L^2(0,t;L^\infty)}^2
 \le\eta Z^2+C_\eta t P(K)\mathfrak D_{n-1}(t),\\
 \int_0^t\!\left|\int_{\mathbb R}I_2\widetilde\theta^{(n)}\,\mathrm{d}x\right|\mathrm{d}\tau
 &\le C Z\|\widetilde u_x^{(n-1)}\|_{L^2_tL^2_x}
 \|u_x^{(n)}+u_x^{(n-1)}\|_{L^2_tL^\infty_x}
 \le\eta Z^2+C_\eta t^{1/2}P(K)\mathfrak D_{n-1}(t),\\
 \int_0^t\!\left|\int_{\mathbb R}I_3\widetilde\theta^{(n)}\,\mathrm{d}x\right|\mathrm{d}\tau
 &\le C \mathfrak D_{n-1}(t)^{1/2}Z\|u_x^{(n)}\|_{L^1_tL^\infty_x}
 \le\eta Z^2+C_\eta t^{3/2}P(K)\mathfrak D_{n-1}(t),\\
 \int_0^t\!\left|\int_{\mathbb R}I_4\widetilde\theta^{(n)}\,\mathrm{d}x\right|\mathrm{d}\tau
 &\le Ct^{1/2}Z\|\widetilde u_x^{(n-1)}\|_{L^2_tL^2_x}
 \le\eta Z^2+C_\eta t \mathfrak D_{n-1}(t).
\end{align*}
Here the time intervals in abbreviated norms are $(0,t)$, and all spatial
norms are taken over $\mathbb R$. Consequently,
\begin{equation}\label{difference-F-final}
\begin{aligned}
 \int_0^t\left|\int_{\mathbb R}
 \widetilde F^{(n-1)}\widetilde\theta^{(n)}\,\mathrm{d}x\right|\mathrm{d}\tau
 \le{}&4\eta\|\widetilde\theta^{(n)}\|_{L^\infty(0,t;L^2)}^2\\
 &+C_\eta\bigl(t^{1/2}+t+t^{3/2}\bigr)
 P(K)\mathfrak D_{n-1}(t).
\end{aligned}
\end{equation}

Because $b^{(n)}$ and $c^{(n)}$ are bounded above and below by positive
constants, integration of \eqref{difference-theta-energy}, followed by
\eqref{difference-ct-term}, \eqref{difference-b-term},
\eqref{difference-c-term}, and \eqref{difference-F-final}, gives
\begin{align}
 &\|\widetilde\theta^{(n)}\|_{L^\infty(0,T;L^2)}^2
 +\|\widetilde\theta_x^{(n)}\|_{L^2(0,T;L^2)}^2\nonumber\\
 &\quad\le C\bigl(\eta+T^{1/3}P(K)\bigr)
 \left[
 \|\widetilde\theta^{(n)}\|_{L^\infty(0,T;L^2)}^2
 +\|\widetilde\theta_x^{(n)}\|_{L^2(0,T;L^2)}^2\right]\nonumber\\
 &\qquad
 +C_\eta\bigl(T^{1/2}+T+T^{3/2}\bigr)
 P(K)\mathfrak D_{n-1}(T).
\label{difference-theta-pre-final}
\end{align}
Choose first $\eta>0$ so small that $C\eta<1/4$, and then
$T_1\le\min\{T_0,1\}$ so small that
$CT_1^{1/3}P(K)<1/4$.  The first term on the right-hand side
can then be absorbed.  Notice the order of these choices:
$\eta$ is fixed \emph{before} the time interval is reduced, so the
constant $C_\eta$ cannot interfere with uniformity in the iteration
index.  We obtain
\begin{equation}\label{difference-theta-est}
 \|\widetilde\theta^{(n)}\|_{L^\infty(0,T_1;L^2)}^2
 +\|\widetilde\theta_x^{(n)}\|_{L^2(0,T_1;L^2)}^2
 \le C T_1^{1/2}P(K)\mathfrak D_{n-1}(T_1).
\end{equation}
Combining \eqref{difference-v-est}, \eqref{difference-u-est}, and
\eqref{difference-theta-est}, and reducing $T_1$ once more if necessary,
we obtain
\[
 \mathfrak D_n(T_1)\le C T_1^{1/2}P(K)\mathfrak D_{n-1}(T_1)
 \le\frac12\mathfrak D_{n-1}(T_1).
\]
This proves the lemma.
\end{proof}

We next pass to the limit. By Lemma~\ref{lem:contraction-local},
$\sum_{n\ge1}\mathfrak D_n(T_1)^{1/2}<\infty$.
Thus $(v^{(n)}-\bar v,u^{(n)},
\theta^{(n)}-\bar\theta)$ is Cauchy in
$L^\infty(0,T_1;L^2)^3$, and
$u_x^{(n)},\theta_x^{(n)}$ are Cauchy in
$L^2(0,T_1;L^2)$.
Denote the corresponding limits by
$(v-\bar v,u,\theta-\bar\theta)$,
$u_x$ and $\theta_x$. The uniform estimate
\eqref{standing-uniform-bound} furnishes weak-star convergence in
$L^\infty_tH^1_x$ and weak convergence in the stated $L^2$
derivative spaces; the strong limits identify all weak derivatives.
In particular,
\begin{equation}\label{limit-weak-regularity}
\begin{aligned}
 &(v-\bar v,u,\theta-\bar\theta)\in L^\infty(0,T_1;H^1),\\
 &(u_x,\theta_x)\in L^2(0,T_1;H^1),\quad
 v_t\in L^2(0,T_1;H^1),\\
 &(u_t,\theta_t)\in L^2(0,T_1;L^2).
\end{aligned}
\end{equation}
For each $z=v,u,\theta$, the interpolation inequality and the
uniform $L^\infty_tH^1_x$ bound imply
\begin{equation}\label{uniform-convergence-iterates}
\|z^{(n)}-z\|_{L^\infty((0,T_1)\times\mathbb R)}^2
\le C\|z^{(n)}-z\|_{L^\infty_tL^2_x}
     \|\partial_x(z^{(n)}-z)\|_{L^\infty_tL^2_x}\to0.
\end{equation}
It follows that $p^{(n)}$, $c^{(n)}$, $\mu/v^{(n)}$ and
$\kappa/v^{(n)}$ converge uniformly to their limiting
coefficients. We do \emph{not} claim strong $L^2$ convergence of all
coefficient spatial derivatives: for example, $v_x^{(n)}$ is only
known to converge weakly in the relevant $L^2$ spaces. No such
strong convergence is needed here. The velocity and temperature
equations are retained in divergence form, so the limiting viscous
and heat-flux terms require only strong convergence of the fluxes;
the pressure derivative is interpreted as the distributional
derivative of $p(v,\theta)$, identified using convergence of the
pressure itself. In particular, the products
$a^{(n-1)}u_x^{(n)}$ and $b^{(n-1)}\theta_x^{(n)}$
converge strongly in $L^2_tL^2_x$, while
$c^{(n-1)}\theta_t^{(n)}\rightharpoonup c(v,\theta)\theta_t$
weakly in $L^2_tL^2_x$. The latter follows by splitting the
coefficient difference, which tends to zero in $L^\infty_{t,x}$,
from the weak convergence of $\theta_t^{(n)}$. Moreover,
\[
 (u_x^{(n)})^2\longrightarrow u_x^2
             \quad\text{in }L^1(0,T_1;L^1),
\]
by the $L^2_tL^2_x$ strong convergence. The sequence of nonlinear
temperature sources is bounded in $L^2_tL^2_x$ by
\eqref{f-L2-time} and converges to the expected source in
$L^1_tL^1_x$ on bounded cylinders, since multiplication maps
$L^2\times L^2$ continuously into $L^1$ and the constitutive
coefficients converge uniformly. Its weak $L^2_tL^2_x$ limit is
therefore the expected source as a distribution, and hence as an
$L^2$ function. Testing the iteration
equations against compactly supported smooth functions and passing
to these limits establishes the nonlinear system almost everywhere.

The initial trace is not inferred merely from pointwise
convergence. For $u^{(n)}$ and
$\theta^{(n)}-\bar\theta$, the uniform bounds in
$L^2(0,T_1;H^2)\cap H^1(0,T_1;L^2)$,
together with the continuous trace map into $H^1$,
identify their common initial traces under weak
convergence.

For the specific volume, passing to the limit in
\[
v^{(n)}(t)=v_0+
\int_0^t u_x^{(n-1)}(\tau)\,\mathrm{d}\tau
\]
gives $v(0)=v_0$. By the Hilbert-space continuity
theorem,
\[
(u,\theta-\bar\theta)\in C([0,T_1];H^1).
\]
Since $v_t=u_x\in L^2(0,T_1;H^1)$, we also have
\[
v-\bar v\in C([0,T_1];H^1).
\]
The pointwise positivity bounds pass to the limit
through \eqref{uniform-convergence-iterates}.
Thus
\begin{equation}\label{local-limit-regularity}
\begin{aligned}
 &(v-\bar v,u,\theta-\bar\theta)\in C([0,T_1];H^1),\\
 &(u_x,\theta_x)\in L^2(0,T_1;H^1),\quad
 v_t\in L^2(0,T_1;H^1),\\
 &(u_t,\theta_t)\in L^2(0,T_1;L^2),\qquad
 \frac{c_0}{2}\le v,\theta\le 2C_0.
\end{aligned}
\end{equation}
This construction uses a lower-order contraction but recovers the
full strong-solution regularity from uniform estimates.  In particular,
there are two different uses of time regularity which should not be
confused: the scalar linear step of Lemma~\ref{lem:local-linear} is
proved without $\omega_t$ or $\sigma_t$, whereas the nonlinear Picard
difference energy uses $(c^{(n)})_t\in L^2_tL^2_x$.
The latter is a consequence of the iterates' equations and of
$e\in C^2$, as shown in \eqref{ct-uniform-difference}; it is not an
additional assumption on the constitutive laws or the coefficients
of the abstract linear lemma.

\subsection{Uniqueness}\label{subsec:local-uniqueness}

Let $(v_i,u_i,\theta_i)$, $i=1,2$, be two strong solutions on $[0,T_1]$
in the class \eqref{local-limit-regularity}, satisfying the same initial
data and the same positive upper and lower bounds.  Set
\[
 V=v_1-v_2,\qquad U=u_1-u_2,\qquad \Theta=\theta_1-\theta_2.
\]
Write
\[
 a_i=\frac{\mu}{v_i},\qquad b_i=\frac{\kappa}{v_i},\qquad
 c_i=c(v_i,\theta_i),\qquad p_i=p(v_i,\theta_i),
\]
and
\[
 F_i=\frac{\mu}{v_i}u_{ix}^2
 -\theta_i p_\theta(v_i,\theta_i)u_{ix}.
\]
Then
\begin{equation}\label{uniq-difference-system}
\left\{
\begin{aligned}
 &V_t-U_x=0,\\
 &U_t-(a_1U_x)_x=-(p_1-p_2)_x+\bigl((a_1-a_2)u_{2x}\bigr)_x,\\
 &c_1\Theta_t-(b_1\Theta_x)_x
 =(F_1-F_2)+\bigl((b_1-b_2)\theta_{2x}\bigr)_x
 -(c_1-c_2)\theta_{2t},
\end{aligned}
\right.
\end{equation}
with $(V,U,\Theta)|_{t=0}=0$.

The mean value theorem gives
\begin{equation}\label{uniq-coeff-diff}
 |p_1-p_2|+|c_1-c_2|
 \le C(|V|+|\Theta|),
 \qquad
 |a_1-a_2|+|b_1-b_2|\le C|V|.
\end{equation}
Testing the first equation in \eqref{uniq-difference-system} by $V$, the
second by $U$, and the third by $\Theta$, and arguing exactly as in the
contraction estimate, we obtain, for every sufficiently small $\eta>0$,
\begin{equation}\label{uniq-energy-pre}
\begin{aligned}
 &\frac{\mathrm{d}}{\mathrm{d}t}
 \left(\|V\|^2+\|U\|^2+\int_{\mathbb R}c_1\Theta^2\,\mathrm{d}x\right)
 +\|U_x\|^2+\|\Theta_x\|^2\\
 &\qquad\le C\,G(t)
 \left(\|V\|^2+\|U\|^2+\|\Theta\|^2\right),
\end{aligned}
\end{equation}
where one may take
\begin{equation}\label{uniq-G}
\begin{aligned}
 G(t):=1
 &+\|u_{1x}(t)\|_{L^\infty}^2
 +\|u_{2x}(t)\|_{L^\infty}^2
 +\|\theta_{2x}(t)\|_{L^\infty}^2\\
 &+\|(c_1)_t(t)\|_{L^2}^{4/3}
 +\|\theta_{2t}(t)\|_{L^2}^{4/3}.
\end{aligned}
\end{equation}
We record the two terms specific to the temperature equation.  First,
\[
 \left|\int (c_1)_t\Theta^2\,\mathrm{d}x\right|
 \le \eta\|\Theta_x\|^2
 +C_\eta\|(c_1)_t\|_{L^2}^{4/3}\|\Theta\|^2.
\]
Second, by \eqref{uniq-coeff-diff} and the one-dimensional interpolation
inequality,
\[
 \left|\int(c_1-c_2)\theta_{2t}\Theta\,\mathrm{d}x\right|
 \le \eta\|\Theta_x\|^2
 +C_\eta\|\theta_{2t}\|_{L^2}^{4/3}
 \bigl(\|V\|^2+\|\Theta\|^2\bigr).
\]
For $F_1-F_2$, we write directly
\begin{align*}
 F_1-F_2={}&\left(\frac{\mu}{v_1}-\frac{\mu}{v_2}\right)u_{1x}^2
 +\frac{\mu}{v_2}(u_{1x}+u_{2x})U_x\\
 &-\Bigl[\theta_1p_\theta(v_1,\theta_1)
 -\theta_2p_\theta(v_2,\theta_2)\Bigr]u_{1x}
 -\theta_2p_\theta(v_2,\theta_2)U_x.
\end{align*}
By the mean value theorem, the two coefficient differences are bounded by
$C(|V|+|\Theta|)$.  Hence its contribution is bounded by
\[
 \eta\|U_x\|^2
 +C_\eta\bigl(1+\|u_{1x}\|_{L^\infty}^2
 +\|u_{2x}\|_{L^\infty}^2\bigr)
 \bigl(\|V\|^2+\|\Theta\|^2\bigr).
\]
For example, the viscosity-coefficient term is bounded by
\[
 \left|\int(a_1-a_2)u_{2x}U_x\,\mathrm{d}x\right|
 \le\eta\|U_x\|_2^2
  +C_\eta\|u_{2x}\|_\infty^2\|V\|_2^2,
\]
and the analogous temperature-diffusion term uses
$\|\theta_{2x}\|_\infty^2\|V\|_2^2$.
The pressure difference contributes
$\eta\|U_x\|_2^2+C_\eta(\|V\|_2^2+\|\Theta\|_2^2)$.
For the transport part of the first equation,
$2|\langle V,U_x\rangle|\le\eta\|U_x\|_2^2
+C_\eta\|V\|_2^2$. These terms, together with the
explicit temperature estimates above, give \eqref{uniq-energy-pre}.

By \eqref{local-limit-regularity} and the one-dimensional Sobolev embedding,
$u_{ix},\theta_{ix}\in L^2(0,T_1;L^\infty)$.  Moreover,
$(c_1)_t\in L^2(0,T_1;L^2)$ and
$\theta_{2t}\in L^2(0,T_1;L^2)$.  Hence $G\in L^1(0,T_1)$.  Since $c_1$
is bounded above and below by positive constants and the initial difference
vanishes, Gronwall's inequality applied to \eqref{uniq-energy-pre} gives
\[
 V=U=\Theta\equiv0\qquad\text{on }[0,T_1]\times\mathbb R.
\]
Therefore the local strong solution is unique.

\section{Relative Energy Identity and Local Coercivity: Proof of Proposition \ref{prop:relative-energy}}
\label{sec:relative-energy}

In this section, we prove Proposition \ref{prop:relative-energy}. The main point is that,
once the thermodynamic compatibility relation is imposed, the basic
zeroth-order energy structure can be derived directly from the Gibbs
relation. In particular, no separated representation of the pressure with
respect to $v$ and $\theta$ is required.

Throughout this section, the assumptions and notation of Section~\ref{sec:intro} are in force.

\begin{proof}[Proof of Proposition \ref{prop:relative-energy}]

\medskip
\noindent\textit{Step 1. Entropy associated with the constitutive laws.}
We first show that the thermodynamic compatibility condition
\eqref{thermo-relation} gives rise to a specific entropy $s=s(v,\theta)$.
Indeed, the Gibbs relation \eqref{Gibbs} gives \eqref{entropy-derivatives}.
These two identities are compatible. In fact, differentiating
\eqref{thermo-relation} with respect to $\theta$ yields
\[
e_{v\theta}(v,\theta)
=
\theta p_{\theta\theta}(v,\theta),
\]
and hence
\[
\partial_\theta p_\theta(v,\theta)
=
p_{\theta\theta}(v,\theta)
=
\partial_v\left(\frac{e_\theta(v,\theta)}{\theta}\right).
\]
Since $\mathbb{R}_+\times\mathbb{R}_+$ is simply connected, there exists,
up to an additive constant, a function $s\in C^2(\mathbb{R}_+\times
\mathbb{R}_+)$ satisfying \eqref{entropy-derivatives}.

Using $v_t=u_x$ and the temperature equation \eqref{temperature-equation},
we obtain
\begin{align*}
s_t
&=
s_vv_t+s_\theta\theta_t\\
&=
p_\theta(v,\theta)u_x
+\frac{e_\theta(v,\theta)}{\theta}\theta_t\\
&=
\frac{1}{\theta}
\left(\frac{\kappa\theta_x}{v}\right)_x
+
\frac{\mu u_x^2}{v\theta}.
\end{align*}
Since
\[
\frac{1}{\theta}
\left(\frac{\kappa\theta_x}{v}\right)_x
=
\left(\frac{\kappa\theta_x}{v\theta}\right)_x
+
\frac{\kappa\theta_x^2}{v\theta^2},
\]
we arrive at the entropy balance
\begin{equation}\label{entropy-balance-sec3}
s_t
=
\left(\frac{\kappa\theta_x}{v\theta}\right)_x
+
\frac{\mu u_x^2}{v\theta}
+
\frac{\kappa\theta_x^2}{v\theta^2}.
\end{equation}

\medskip
\noindent\textit{Step 2. Relative-energy identity.}
Recall the relative thermodynamic potential \eqref{relative-H}.
We now combine the total-energy equation \eqref{NSc} with \eqref{entropy-balance-sec3}. Since
\[
\bar p(v-\bar v)_t=\bar p\,u_x=(\bar p\,u)_x,
\]
subtracting $\bar\theta$ times \eqref{entropy-balance-sec3} from
\eqref{NSc}, and adding the preceding identity, gives
\begin{equation}\label{relative-local-balance}
\left(
\frac12u^2
+
\mathcal H(v,\theta\mid\bar v,\bar\theta)
\right)_t
+
\left[
u(p-\bar p)
-\frac{\mu uu_x}{v}
-\frac{\kappa\theta_x}{v}
+\frac{\bar\theta\kappa\theta_x}{v\theta}
\right]_x
+
\bar\theta
\left(
\frac{\mu u_x^2}{v\theta}
+
\frac{\kappa\theta_x^2}{v\theta^2}
\right)
=0.
\end{equation}
We justify passage to the whole line explicitly. Fix a finite time
interval $[0,T]$ on which the strong solution has positive, bounded
$v,\theta$ and its perturbation belongs to $L^\infty_tH^1_x$,
with $u_x,\theta_x\in L^2_tL^2_x$. Choose
$\chi\in C_c^\infty(\mathbb R)$, $0\le\chi\le1$,
$\chi=1$ on $[-1,1]$, and put $\chi_R(x)=\chi(x/R)$.
Write the flux in \eqref{relative-local-balance} as $J=J_1+J_2$,
where
\[
 J_1=u(p-\bar p)-\frac{\mu uu_x}{v},\qquad
 J_2=-\frac{\kappa\theta_x}{v}
         +\frac{\bar\theta\kappa\theta_x}{v\theta}.
\]
The mean value theorem on the compact state region gives
$\|p-\bar p\|_2\le C\|(v-\bar v,\theta-\bar\theta)\|_2$.
Consequently $J_1\in L^1(0,T;L^1(\mathbb R))$ by
Cauchy--Schwarz and $J_2\in L^1(0,T;L^2(\mathbb R))$.
Since $\|\chi_R'\|_\infty\le C/R$ and
$\|\chi_R'\|_2\le CR^{-1/2}$, we have
\begin{align*}
 \left|\int_0^T\!\int_\mathbb R J\chi_R'\,\mathrm{d}x\,\mathrm{d}t\right|
 &\le \frac C R\|J_1\|_{L^1_tL^1_x}
    +\frac C{\sqrt R}\|J_2\|_{L^1_tL^2_x}
 \longrightarrow0.
\end{align*}
The relative-energy density is integrable by Taylor's formula
near the reference state (or by compact-state boundedness and
$H^1$ perturbations); the nonnegative entropy production is also
integrable over $[0,T]$. Multiplying the local balance by $\chi_R$,
integrating over $[0,t]\times\mathbb R$, and sending $R\to\infty$
by dominated convergence therefore gives
\eqref{relative-energy-identity} for every $0\le t\le T$.
A standard approximation yields the same identity at the stated
strong regularity.
Since $v>0$, $\theta>0$, $\mu>0$, $\kappa>0$, and $\bar\theta>0$, the
dissipation in \eqref{relative-energy-identity} is nonnegative.

It is worth emphasizing that the identity
\eqref{relative-energy-identity} follows solely from the Gibbs relation and the
thermodynamic compatibility condition. No separated form such as
$p(v,\theta)=\alpha(v)\beta(\theta)+\gamma(v)$ is needed.

\medskip
\noindent\textit{Step 3. Local nonnegativity and quadratic coercivity of
$\mathcal H$.}
We next study the sign of the thermodynamic part of the relative energy.
From \eqref{relative-H},
\eqref{entropy-derivatives}, and
\eqref{thermo-relation}, we compute
\begin{align}
\mathcal H_v(v,\theta\mid\bar v,\bar\theta)
&=
e_v(v,\theta)+\bar p-\bar\theta s_v(v,\theta)
\nonumber\\
&=
(\theta-\bar\theta)p_\theta(v,\theta)
-p(v,\theta)+\bar p,
\label{H-v-sec3}
\end{align}
and
\begin{align}
\mathcal H_\theta(v,\theta\mid\bar v,\bar\theta)
&=
e_\theta(v,\theta)
-\bar\theta s_\theta(v,\theta)
\nonumber\\
&=
e_\theta(v,\theta)
\left(1-\frac{\bar\theta}{\theta}\right).
\label{H-theta-sec3}
\end{align}
Therefore,
\begin{equation}\label{H-critical-sec3}
\mathcal H(\bar v,\bar\theta\mid\bar v,\bar\theta)=0,
\qquad
\nabla_{v,\theta}\mathcal H
(\bar v,\bar\theta\mid\bar v,\bar\theta)=0.
\end{equation}

Differentiating once more gives
\begin{align}
\mathcal H_{vv}
&=
(\theta-\bar\theta)p_{v\theta}(v,\theta)
-p_v(v,\theta),
\label{H-vv-sec3}\\
\mathcal H_{v\theta}
&=
(\theta-\bar\theta)p_{\theta\theta}(v,\theta),
\label{H-vtheta-sec3}\\
\mathcal H_{\theta\theta}
&=
e_{\theta\theta}(v,\theta)
\left(1-\frac{\bar\theta}{\theta}\right)
+
\frac{\bar\theta}{\theta^2}e_\theta(v,\theta).
\label{H-thetatheta-sec3}
\end{align}
Hence, at the equilibrium state,
\begin{equation}\label{H-hessian-sec3}
D^2_{v,\theta}
\mathcal H(\bar v,\bar\theta\mid\bar v,\bar\theta)
=
\begin{pmatrix}
-p_v(\bar v,\bar\theta) & 0\\[1mm]
0 &
\dfrac{e_\theta(\bar v,\bar\theta)}{\bar\theta}
\end{pmatrix}.
\end{equation}

Assume now the mechanical stability condition \eqref{mechanical-stability}.
Together with $e_\theta(\bar v,\bar\theta)>0$, this implies that the Hessian
in \eqref{H-hessian-sec3} is positive definite. Set
\begin{equation}\label{lambda-sec3}
\lambda_0
:=
\min\left\{
-p_v(\bar v,\bar\theta),
\frac{e_\theta(\bar v,\bar\theta)}{\bar\theta}
\right\}>0.
\end{equation}
For simplicity of notation, we denote $D^2_{v,\theta}
\mathcal H( v,\theta\mid\bar v,\bar\theta)$ by $D^2\mathcal H(v,\theta)$. Since $\mathcal H\in C^2$, the Hessian is continuous. Therefore, there
exists $\delta>0$ such that
\begin{equation}\label{Hessian-lower-sec3}
\xi^{\mathrm T}D^2\mathcal H(v,\theta)\xi
\ge
\frac{\lambda_0}{2}|\xi|^2
\end{equation}
for every $\xi\in\mathbb R^2$ whenever
\[
|v-\bar v|+|\theta-\bar\theta|\le\delta.
\]
On the same neighborhood, there exists $\Lambda_0>0$ such that
\begin{equation}\label{Hessian-upper-sec3}
\xi^{\mathrm T}D^2\mathcal H(v,\theta)\xi
\le
\Lambda_0|\xi|^2.
\end{equation}

Let
\[
z:=
\begin{pmatrix}
v-\bar v\\
\theta-\bar\theta
\end{pmatrix}.
\]
By \eqref{H-critical-sec3} and Taylor's formula with integral remainder,
\begin{equation}\label{Taylor-H-sec3}
\mathcal H(v,\theta\mid\bar v,\bar\theta)
=
\int_0^1
(1-\tau)
z^{\mathrm T}
D^2\mathcal H
\bigl(
\bar v+\tau(v-\bar v),
\bar\theta+\tau(\theta-\bar\theta)
\bigr)
z
\,\mathrm{d}\tau.
\end{equation}
Using \eqref{Hessian-lower-sec3} and \eqref{Hessian-upper-sec3}, Taylor's formula yields the two-sided estimate \eqref{coercivity-relative-energy}, possibly after changing the constants $c_*$ and $C_*$.
Thus $\mathcal H$ is locally nonnegative and quadratically coercive near
$(\bar v,\bar\theta)$, with equality only at the equilibrium.

We stress that \eqref{coercivity-relative-energy} is a local statement. The
conditions $p_v(\bar v,\bar\theta)<0$ and
$e_\theta(\bar v,\bar\theta)>0$ do not, by themselves, imply that
$\mathcal H$ is globally nonnegative on all of
$\mathbb R_+\times\mathbb R_+$. Such a conclusion would require additional
global convexity or thermodynamic stability assumptions. For the
small-perturbation global theory developed below, however, the local
coercivity \eqref{coercivity-relative-energy} is sufficient.

This completes the proof of Proposition \ref{prop:relative-energy}.
\end{proof}

The relative-energy identity \eqref{relative-energy-identity}, together with
the local coercivity estimate \eqref{coercivity-relative-energy}, provides the basic
zeroth-order estimate required for the global continuation argument. In
the next section, these estimates will be combined with higher-order
energy bounds and a smallness assumption on the initial perturbation to
prove the global existence of strong solutions.

\section{Global Existence for Small Perturbations: Proof of Theorem \ref{thm:global-small}}
\label{sec:global}

\begin{lemma}[Uniform local restart and continuation alternative]
\label{lem:uniform-restart}
Fix positive numbers $B,m,M$ with $m<M$, and fix the constitutive
laws, the equilibrium and the transport coefficients. There exists
$\tau=\tau(B,m,M)>0$ such that any strong-solution initial state at
a time $t_0\ge0$ satisfying
\begin{equation}\label{restart-data-bounds}
 \|(v(t_0)-\bar v,u(t_0),\theta(t_0)-\bar\theta)\|_{H^1}\le B,
 \qquad m\le v(t_0,x),\theta(t_0,x)\le M
\end{equation}
admits a strong solution on $[t_0,t_0+\tau]$, with $\tau$ independent
of $t_0$ and of the particular state satisfying
\eqref{restart-data-bounds}. Consequently, if a maximal solution has
finite existence time $T_*$, then it is impossible for both its
$H^1$ norm and its positive upper and lower pointwise bounds to remain
uniformly controlled on $[0,T_*)$.
\end{lemma}
\begin{proof}
Apply the construction in Section~\ref{sec:local} with $c_0=m$,
$C_0=M$. In the proof of Lemma~\ref{lem:uniform-local}, the radius
$K$ is chosen from a fixed multiple of $1+B$; the coefficient bounds
are then taken on the common compact rectangle
$[m/4,4M]^2$. Every small-time requirement in the induction,
pointwise comparison and Picard contraction uses only these numbers
and the constitutive bounds, so one common $\tau>0$ can be chosen.
Time translation does not change the coefficients of the autonomous
system. For the final assertion, restart at a time
$t_0\in(T_*-\tau/2,T_*)$ and use uniqueness on the overlap to obtain a
solution beyond $T_*$, a contradiction.
\end{proof}

\begin{proof}[proof of Theorem \ref{thm:global-small}]
Let $T_*>0$ be the maximal existence time of the strong solution furnished
by Theorem \ref{thm:local-large}. We prove that $T_*=\infty$ by deriving
uniform a priori estimates.

\medskip
\noindent\textit{Justification of the energy tests.}
Fix $T<T_*$. The regularity in Theorem~\ref{thm:local-large} gives
$u,\theta\in L^2(0,T;H^2)\cap C([0,T];H^1)$ and
$u_t,\theta_t\in L^2(0,T;L^2)$, while
$v\in C([0,T];H^1)$ and $v_t=u_x\in L^2(0,T;H^1)$.
The constitutive coefficients and their derivatives with respect to the
state variables are bounded on the compact thermodynamic neighborhood
used below. In particular, for
$w=\mu v_x/v$ we have
\[
 w_t=\left(\frac{\mu u_x}{v}\right)_x
 =\frac{\mu u_{xx}}{v}-\frac{\mu u_xv_x}{v^2}
 \in L^2(0,T;L^2),
\]
because $u_x\in L^2(0,T;L^\infty)$ in one space dimension and
$v_x\in L^\infty(0,T;L^2)$. Thus $w$ has a continuous $L^2$ trace,
and the endpoint pairing $\int_{\mathbb R}uw$ is meaningful.
Testing the momentum equation with $-u_{xx}$ and the normalized
thermal equation with $-\theta_{xx}$ is legitimate by the stated strong
regularity (alternatively, by time regularization followed by a limit).
All spatial integrations by parts can first be performed with a standard
cutoff $\chi_R(x)=\chi(x/R)$, then passed to $R\to\infty$; the
fluxes encountered below are locally time-integrable $L^1_x$ products,
so their cutoff remainders tend to zero. The calculations below are
therefore a priori identities on every finite interval $[0,T]$ rather
than assumptions of additional classical smoothness.

\medskip
\noindent\textit{Step 1. Bootstrap interval and thermodynamic bounds.}
Choose $\delta>0$ once and for all, small enough that
$\delta<\frac14\min\{\bar v,\bar\theta\}$, the local coercivity
\eqref{coercivity-relative-energy} holds on the corresponding
thermodynamic neighborhood, and the coefficient bounds below hold.
Set
\[
 \mathcal B(t):=\|v(t)-\bar v\|_\infty+\|u(t)\|_\infty
                 +\|\theta(t)-\bar\theta\|_\infty.
\]
By $C_tH^1_x$ regularity and the one-dimensional Sobolev embedding,
$\mathcal B$ is continuous. Taking the initial $H^1$ perturbation
small enough ensures $\mathcal B(0)<\delta/2$. Define the endpoint
of the first bootstrap interval by
\[
 T_{\delta}:=\sup\bigl\{t<T_*:\ \mathcal B(s)<\delta
                 \text{ for every }0\le s\le t\bigr\}.
\]
Then $T_\delta>0$, and, on each compact interval $[0,T]$ with
$T<T_\delta$, we have the bootstrap bound
\begin{equation}\label{global-bootstrap}
|v(t,x)-\bar v|+|u(t,x)|+|\theta(t,x)-\bar\theta|
\le \delta,
\qquad
(t,x)\in[0,T]\times\mathbb R.
\end{equation}
All constants in the ensuing estimates depend on this fixed
$\delta$, not on $T$. We will show that
$\mathcal B(t)\le\delta/2$ throughout the bootstrap interval; its
continuity then excludes an endpoint $T_\delta<T_*$. This is the
continuity argument underlying the global estimates.

Since $p_v(\bar v,\bar\theta)<0$ and
$e_\theta(\bar v,\bar\theta)>0$, by continuity we may choose $\delta$ so
small that throughout the neighborhood determined by
\eqref{global-bootstrap},
\begin{equation}\label{global-coeff-bounds}
\frac{\bar v}{2}\le v\le\frac{3\bar v}{2},
\qquad
\frac{\bar\theta}{2}\le\theta\le\frac{3\bar\theta}{2},
\end{equation}
and, for some constants $c_1,c_2>0$,
\begin{equation}\label{global-pv-cv-bounds}
-p_v(v,\theta)\ge c_1>0,
\qquad
e_\theta(v,\theta)\ge c_2>0.
\end{equation}
Moreover, all the coefficients
\[
p_v,\quad p_\theta,\quad e_\theta,\quad
\frac1{e_\theta},\quad \frac1v,\quad \frac1\theta
\]
are uniformly bounded on this neighborhood.

\medskip
\noindent\textit{Step 2. Zeroth-order relative-energy estimate.}
Write $\rho_0=\|(v_0-\bar v,u_0,\theta_0-\bar\theta)\|_{H^1}$,
so that $\rho_0\le\varepsilon_0$. In Steps~2--5 we keep the actual
initial size $\rho_0$ in every quantitative estimate: the initial
terms are bounded by $\rho_0^2$, and the estimates are derived
successively in the order displayed below. The threshold
$\varepsilon_0$ is used only to ensure that the nonlinear higher powers
are controlled and the bootstrap closes. This distinction provides
the initial-norm right-hand side in
\eqref{global-quantitative-estimate}.
By Proposition \ref{prop:relative-energy}, the relative thermodynamic potential
$\mathcal H(v,\theta\mid\bar v,\bar\theta)$ satisfies, for $\delta$ small
enough,
\begin{equation}\label{global-H-coercivity}
c
\bigl(
|v-\bar v|^2+|\theta-\bar\theta|^2
\bigr)
\le
\mathcal H(v,\theta\mid\bar v,\bar\theta)
\le
C
\bigl(
|v-\bar v|^2+|\theta-\bar\theta|^2
\bigr).
\end{equation}
The relative-energy identity \eqref{relative-energy-identity}, together with \eqref{global-coeff-bounds}, shows that the dissipation coefficients are bounded above and below by positive constants. Furthermore, if $\varepsilon_0$ in \eqref{ini} is sufficiently
small, the Sobolev embedding $H^1(\mathbb R)\hookrightarrow L^\infty(\mathbb
R)$ ensures that the initial state lies in the coercivity neighborhood.
Therefore,
\begin{equation}\label{L2-bound}
\sup_{0\le t\le T}
\|(v-\bar v,u,\theta-\bar\theta)(t)\|^2
+
\int_0^T
\|(u_x,\theta_x)(t)\|^2\,\mathrm{d}t
\le
C\rho_0^2.
\end{equation}
Here and below, $C>0$ is independent of $T<T_*$.

\medskip
\noindent\textit{Step 3. Estimate of $v_x$.}
Set
\begin{equation}\label{effective-vx}
w:=\frac{\mu v_x}{v}.
\end{equation}
Since $v_t=u_x$, we have
\[
w_t
=
\left(\frac{\mu u_x}{v}\right)_x.
\]
Hence the momentum equation \eqref{NS1b} can be written as
\[
u_t+p_x=w_t.
\]
Multiplying this equation by $w$ and using
\[
u_t w
=
(uw)_t
-
\left(\frac{\mu uu_x}{v}\right)_x
+
\frac{\mu u_x^2}{v},
\]
we obtain
\begin{equation}\label{vx-identity}
\begin{aligned}
\frac12(w^2)_t
+
\frac{-\mu p_v(v,\theta)}{v}v_x^2
={}&
(uw)_t
-
\left(\frac{\mu uu_x}{v}\right)_x
+
\frac{\mu u_x^2}{v}
\\
&+
\frac{\mu p_\theta(v,\theta)}{v}v_x\theta_x.
\end{aligned}
\end{equation}
The sign in \eqref{vx-identity} is essential: since $p_v<0$ on
the bootstrap neighborhood, $-\mu p_v/v$ is bounded below by a
positive constant. Integrating over $\mathbb R\times(0,t)$ first
gives the explicit endpoint identity
\begin{align*}
 \frac12\|w(t)\|^2
 +\int_0^t\!\int_{\mathbb R}
       \frac{-\mu p_v}{v}|v_x|^2\,\mathrm{d}x\,\mathrm{d}\tau
 ={}&\frac12\|w(0)\|^2
 +\int_{\mathbb R}u(t)w(t)\,\mathrm{d}x
 -\int_{\mathbb R}u_0w(0)\,\mathrm{d}x\\
 &+\int_0^t\!\int_{\mathbb R}
    \frac{\mu}{v}u_x^2\,\mathrm{d}x\,\mathrm{d}\tau
 +\int_0^t\!\int_{\mathbb R}
    \frac{\mu p_\theta}{v}v_x\theta_x
       \,\mathrm{d}x\,\mathrm{d}\tau.
\end{align*}
The first endpoint product is absorbed using
$\left|\int uw\right|\le\eta\|w\|^2+C_\eta\|u\|^2$;
at the initial endpoint use
$\left|\int u_0w(0)\right|\le C(\|u_0\|^2+\|v_{0x}\|^2)$.
Also, $\|w\|_2$ and $\|v_x\|_2$ are uniformly equivalent under
\eqref{global-coeff-bounds}. Applying Young's inequality to the mixed
$v_x\theta_x$ term and using
\eqref{global-coeff-bounds}--\eqref{global-pv-cv-bounds}, we infer
\begin{align}
\|v_x(t)\|^2
+
\int_0^t\|v_x(\tau)\|^2\,\mathrm{d}\tau
\le{}&
C\|v_{0x}\|^2
+
C\|u_0\|^2
+
C\|u(t)\|^2
\nonumber\\
&+
C\int_0^t
\left(
\|u_x(\tau)\|^2
+
\|\theta_x(\tau)\|^2
\right)
\,\mathrm{d}\tau
\nonumber\\
&+
\eta
\int_0^t\|v_x(\tau)\|^2\,\mathrm{d}\tau,
\label{vx-pre}
\end{align}
where we used
\[
\left|\int_{\mathbb R}u w\,\mathrm{d}x\right|
\le
\eta\|v_x\|^2+C_\eta\|u\|^2
\]
and
\[
\int_{\mathbb R}|v_x\theta_x|\,\mathrm{d}x
\le
\eta\|v_x\|^2+C_\eta\|\theta_x\|^2.
\]
Taking $\eta>0$ sufficiently small and using \eqref{L2-bound}, we conclude
that
\begin{equation}\label{vx-L2}
\|v_x\|_{L^\infty(0,T;L^2)}^2
+
\|v_x\|_{L^2(0,T;L^2)}^2
\le
C\rho_0^2.
\end{equation}

\medskip
\noindent\textit{Step 4. Estimate of $u_x$.}
Multiplying the momentum equation by $-u_{xx}$ and integrating over
$\mathbb R$, we obtain
\begin{equation}\label{ux-basic}
\frac12\frac{\mathrm{d}}{\mathrm{d}t}\|u_x\|^2
+
\int_{\mathbb R}\frac{\mu}{v}u_{xx}^2\,\mathrm{d}x
=
\int_{\mathbb R}p_xu_{xx}\,\mathrm{d}x
-
\int_{\mathbb R}
\left(\frac{\mu}{v}\right)_xu_xu_{xx}\,\mathrm{d}x.
\end{equation}
Since
\[
p_x=p_vv_x+p_\theta\theta_x,
\]
the boundedness of the constitutive coefficients gives
\begin{equation}\label{px-global}
\|p_x\|^2
\le
C\bigl(\|v_x\|^2+\|\theta_x\|^2\bigr).
\end{equation}
Moreover,
\[
\left|
\left(\frac{\mu}{v}\right)_x
\right|
\le C|v_x|.
\]
Using the one-dimensional interpolation inequality
\[
\|f\|_{L^\infty}^2
\le
2\|f\|\|f_x\|,
\]
we estimate
\begin{align}
\int_{\mathbb R}|v_xu_xu_{xx}|\,\mathrm{d}x
&\le
\|v_x\|
\|u_x\|_{L^\infty}
\|u_{xx}\|
\nonumber\\
&\le
C\|v_x\|
\|u_x\|^{1/2}
\|u_{xx}\|^{3/2}
\nonumber\\
&\le
\eta\|u_{xx}\|^2
+
C_\eta\|v_x\|^4\|u_x\|^2.
\label{ux-nonlinear}
\end{align}
Integrating \eqref{ux-basic} over $(0,t)$, using
\eqref{L2-bound}, \eqref{vx-L2}, and \eqref{px-global}, and then choosing
$\eta$ sufficiently small, yields
\begin{align}
\|u_x(t)\|^2
+
\int_0^t\|u_{xx}(\tau)\|^2\,\mathrm{d}\tau
\le{}&
C\|u_{0x}\|^2
+
C\rho_0^2
\nonumber\\
&+
C
\|v_x\|_{L^\infty(0,T;L^2)}^4
\int_0^t\|u_x(\tau)\|^2\,\mathrm{d}\tau
\nonumber\\
\le{}&
C\rho_0^2+C\rho_0^6.
\end{align}
Since $\rho_0\le\varepsilon_0$ and $\varepsilon_0$ is chosen sufficiently
small, the sixth-order term is absorbed into the quadratic bound. Hence
\begin{equation}\label{ux-L2}
\|u_x\|_{L^\infty(0,T;L^2)}^2
+
\|u_{xx}\|_{L^2(0,T;L^2)}^2
\le
C\rho_0^2.
\end{equation}

\medskip
\noindent\textit{Step 5. Estimate of $\theta_x$.}
This is the step where the dependence of $e_\theta$ on $(v,\theta)$ must be
treated carefully. Instead of differentiating the equation
$e_\theta(v,\theta)\theta_t+\theta p_\theta u_x=(\kappa\theta_x/v)_x
+\mu u_x^2/v$ directly, we first divide it by $e_\theta(v,\theta)$.
Writing
\begin{equation}\label{q-global}
q(v,\theta)
:=
\frac{\kappa}{v\,e_\theta(v,\theta)},
\end{equation}
we obtain
\begin{equation}\label{theta-normalized-global}
\theta_t
-
q(v,\theta)\theta_{xx}
=
-\frac{\kappa v_x}{v^2e_\theta(v,\theta)}\theta_x
+
\frac{\mu}{v e_\theta(v,\theta)}u_x^2
-
\frac{\theta p_\theta(v,\theta)}
{e_\theta(v,\theta)}u_x.
\end{equation}
By \eqref{global-coeff-bounds}--\eqref{global-pv-cv-bounds}, there exists
$q_0>0$ such that
\begin{equation}\label{q-positive-global}
q(v,\theta)\ge q_0>0
\end{equation}
under the bootstrap assumption.

Multiplying \eqref{theta-normalized-global} by $-\theta_{xx}$ and integrating
over $\mathbb R$ gives
\begin{align}
\frac12\frac{\mathrm{d}}{\mathrm{d}t}\|\theta_x\|^2
+
q_0\|\theta_{xx}\|^2
\le{}&
C\int_{\mathbb R}|u_x\theta_{xx}|\,\mathrm{d}x
\nonumber\\
&+
C\int_{\mathbb R}|v_x\theta_x\theta_{xx}|\,\mathrm{d}x
+
C\int_{\mathbb R}|u_x|^2|\theta_{xx}|\,\mathrm{d}x.
\label{theta-H1-global}
\end{align}
For the first term,
\begin{equation}\label{theta-term1}
C\int_{\mathbb R}|u_x\theta_{xx}|\,\mathrm{d}x
\le
\eta\|\theta_{xx}\|^2
+
C_\eta\|u_x\|^2.
\end{equation}
For the second term, the one-dimensional interpolation inequality gives
\begin{align}
C\int_{\mathbb R}|v_x\theta_x\theta_{xx}|\,\mathrm{d}x
&\le
C\|v_x\|
\|\theta_x\|_{L^\infty}
\|\theta_{xx}\|
\nonumber\\
&\le
C\|v_x\|
\|\theta_x\|^{1/2}
\|\theta_{xx}\|^{3/2}
\nonumber\\
&\le
\eta\|\theta_{xx}\|^2
+
C_\eta\|v_x\|^4\|\theta_x\|^2.
\label{theta-term2}
\end{align}
For the last term, the one-dimensional interpolation inequality
$\|u_x\|_{L^4}^2\le C\|u_x\|^{3/2}\|u_{xx}\|^{1/2}$
(which follows by interpolating $L^2$ and $L^\infty$) gives
\begin{align}
C\int_{\mathbb R}|u_x|^2|\theta_{xx}|\,\mathrm{d}x
&\le
C\|\theta_{xx}\|
\|u_x\|_{L^\infty}\|u_x\|
\nonumber\\
&\le
C\|\theta_{xx}\|
\|u_x\|^{3/2}
\|u_{xx}\|^{1/2}.
\label{theta-term3-pre}
\end{align}
Young's inequality yields
\begin{equation}\label{theta-term3}
C\|\theta_{xx}\|
\|u_x\|^{3/2}
\|u_{xx}\|^{1/2}
\le
\eta\|\theta_{xx}\|^2
+
C_\eta\|u_x\|^3\|u_{xx}\|.
\end{equation}
Integrating the last factor in time and using \eqref{L2-bound} and
\eqref{ux-L2},
\begin{align}
\int_0^t
\|u_x\|^3\|u_{xx}\|\,\mathrm{d}\tau
&\le
\|u_x\|_{L^\infty(0,T;L^2)}^2
\left(
\int_0^t\|u_x\|^2\,\mathrm{d}\tau
\right)^{1/2}
\left(
\int_0^t\|u_{xx}\|^2\,\mathrm{d}\tau
\right)^{1/2}
\nonumber\\
&\le
C\rho_0^4.
\label{ux-cubic-global}
\end{align}
The exponents are uniform in $T$: the preceding bound is the
product of an $L^\infty_tL^2_x$ square, an $L^2_tL^2_x$ factor,
and an $L^2_tL^2_x$ factor. No factor growing with $T$ occurs.
Similarly, by \eqref{vx-L2} and \eqref{L2-bound},
\begin{equation}\label{vx-theta-global}
\int_0^t
\|v_x\|^4\|\theta_x\|^2\,\mathrm{d}\tau
\le
\|v_x\|_{L^\infty(0,T;L^2)}^4
\int_0^t\|\theta_x\|^2\,\mathrm{d}\tau
\le
C\rho_0^6.
\end{equation}
Substituting \eqref{theta-term1}--\eqref{vx-theta-global} into
\eqref{theta-H1-global}, taking $\eta>0$ sufficiently small, and using
$\|\theta_{0x}\|^2\le\rho_0^2$, we conclude that
\begin{equation}\label{thetax-L2}
\|\theta_x\|_{L^\infty(0,T;L^2)}^2
+
\|\theta_{xx}\|_{L^2(0,T;L^2)}^2
\le
C\rho_0^2.
\end{equation}
More explicitly, the right-hand side after integration in time is
bounded by
$C\|\theta_{0x}\|^2+C\rho_0^2+
C\rho_0^4+C\rho_0^6$.
The small parameters $\eta$ multiplying $\|\theta_{xx}\|^2$
are selected before $\varepsilon_0$ and are absorbed by $q_0$.
Consequently \eqref{thetax-L2} contains no implicit appeal to the
unknown $\theta_{xx}$ dissipation on its right-hand side.

We emphasize that no derivative of $e_\theta(v,\theta)$ is needed in this
estimate. Thus the variable specific heat introduces only bounded
coefficients in \eqref{theta-normalized-global}, and the $C^2$ regularity of
the constitutive laws is sufficient.

\Needspace{8\baselineskip}
\medskip
\noindent\textit{Step 6. Uniform $H^1$ estimate and closure of the
bootstrap.}
The estimates have been derived in the strictly triangular order
\eqref{L2-bound} $\Rightarrow$ \eqref{vx-L2}
$\Rightarrow$ \eqref{ux-L2} $\Rightarrow$ \eqref{thetax-L2}:
in particular, no bound for $u_{xx}$ or $\theta_{xx}$ is used to
establish the effective-gradient estimate \eqref{vx-L2}.
Recall $\rho_0=\|U(0)\|_{H^1}\le\varepsilon_0$. The estimates
in Steps~2--5 have all been written in terms of this actual initial
size: the zeroth-order and volume-gradient contributions are
$O(\rho_0^2)$, the higher-order velocity contribution is
$O(\rho_0^2+\rho_0^6)$, and the temperature contribution is
$O(\rho_0^2+\rho_0^4+\rho_0^6)$. Thus, uniformly on the
bootstrap interval,
\begin{equation}\label{global-actual-initial-size}
 \sup_{0\le t\le T}\|U(t)\|_{H^1}^2
 +\int_0^T\bigl(\|v_x\|_2^2
                 +\|(u_x,\theta_x)\|_{H^1}^2\bigr)\,\mathrm{d}t
 \le C_\delta(\rho_0^2+\rho_0^4+\rho_0^6)
 \le 3C_\delta\rho_0^2,
\end{equation}
provided $\rho_0\le\varepsilon_0\le1$. The constant $C_\delta$
depends on the fixed neighborhood and constitutive bounds, not on
$T$ or $\rho_0$. In particular, the right-hand side of the global
theorem is proportional to the actual initial perturbation, rather
than just the permissible size threshold.
Combining \eqref{L2-bound}, \eqref{vx-L2}, \eqref{ux-L2}, and
\eqref{thetax-L2}, we obtain
\begin{equation}\label{H1-global}
\sup_{0\le t\le T}
\|(v-\bar v,u,\theta-\bar\theta)(t)\|_{H^1}^2
+
\int_0^T
\|(u_x,\theta_x)(t)\|_{H^1}^2\,\mathrm{d}t
\le
C\rho_0^2.
\end{equation}
For later use, we also verify the global time-derivative regularity
claimed in \eqref{global-time-regularity}, with constants independent
of $T$. Since $v_t=u_x$, the $v_t$ estimate follows immediately from
\eqref{H1-global}. Moreover, the one-dimensional interpolation inequality
and Cauchy--Schwarz in time give
\begin{equation}\label{global-gradient-L2Linfty}
 \int_0^T\bigl(\|u_x(t)\|_{L^\infty}^2
             +\|\theta_x(t)\|_{L^\infty}^2\bigr)\,\mathrm{d}t
 \le C\rho_0^2.
\end{equation}
Consequently, using \eqref{vx-L2}, \eqref{ux-L2}, and
\eqref{thetax-L2},
\begin{align*}
 &\int_0^T\bigl(\|v_xu_x\|_2^2
                    +\|v_x\theta_x\|_2^2
                    +\|u_x^2\|_2^2\bigr)\,\mathrm{d}t\\
 &\quad\le
 \sup_{0\le t\le T}\bigl(\|v_x(t)\|_2^2+\|u_x(t)\|_2^2\bigr)
 \int_0^T\bigl(\|u_x\|_{L^\infty}^2
             +\|\theta_x\|_{L^\infty}^2\bigr)\,\mathrm{d}t
 \le C\rho_0^4.
\end{align*}
Expanding the momentum equation and using
\eqref{theta-normalized-global}, we have
\[
 \begin{aligned}
 u_t&=\frac\mu v u_{xx}-\frac{\mu v_x}{v^2}u_x
                  -p_vv_x-p_\theta\theta_x,\\
 \theta_t&=q(v,\theta)\theta_{xx}
       -\frac{\kappa v_x}{v^2e_\theta}\theta_x
       +\frac\mu{ve_\theta}u_x^2
       -\frac{\theta p_\theta}{e_\theta}u_x.
 \end{aligned}
\]
All coefficients are uniformly bounded on the bootstrap
neighborhood. The estimates already proved therefore yield
\begin{equation}\label{global-ut-thetat}
 \int_0^T\bigl(\|u_t(t)\|_2^2
                  +\|\theta_t(t)\|_2^2\bigr)\,\mathrm{d}t
 \le C(\rho_0^2+\rho_0^4)
 \le C\rho_0^2.
\end{equation}
Together with \eqref{vx-L2} and \eqref{H1-global}, this proves
\eqref{global-quantitative-estimate}--\eqref{global-time-regularity}
on every bootstrap interval, and hence globally after continuation.
In particular, the quantitative global conclusion is obtained from
the previously established $H^1$ and parabolic dissipation bounds,
not from a new low-frequency decay assertion.
By the Sobolev embedding $H^1(\mathbb R)\hookrightarrow L^\infty(\mathbb
R)$,
\begin{equation}\label{Linfty-global}
\sup_{0\le t\le T}
\left(
\|v(t)-\bar v\|_{L^\infty}
+
\|u(t)\|_{L^\infty}
+
\|\theta(t)-\bar\theta\|_{L^\infty}
\right)
\le
C\rho_0.
\end{equation}
Since $\rho_0\le\varepsilon_0$, after fixing $\delta$ and the
coefficient constants we choose the smallness threshold $\varepsilon_0>0$
so that
\begin{equation}\label{bootstrap-close}
C\varepsilon_0\le\frac{\delta}{2}.
\end{equation}
Then \eqref{Linfty-global} strictly improves the bootstrap
assumption \eqref{global-bootstrap}. Since $\mathcal B$ is continuous,
$T_\delta<T_*$ would imply
$\mathcal B(T_\delta)\le\delta/2<\delta$, and hence the bootstrap
interval could be extended beyond $T_\delta$, a contradiction.
Therefore $T_\delta=T_*$, and all the displayed estimates hold for
every $T<T_*$; this argument uses no a priori bound beyond the
provisional bootstrap interval. In particular,
\begin{equation}\label{positive-global}
\frac{\bar v}{2}
\le
v(t,x)
\le
\frac{3\bar v}{2},
\qquad
\frac{\bar\theta}{2}
\le
\theta(t,x)
\le
\frac{3\bar\theta}{2}
\end{equation}
for all $(t,x)\in[0,T]\times\mathbb R$.

\medskip
\noindent\textit{Step 7. Continuation.}
The estimates above are independent of $T<T_*$. Hence
\begin{equation}\label{uniform-global-final}
\sup_{0\le t<T_*}
\|(v-\bar v,u,\theta-\bar\theta)(t)\|_{H^1}
\le
C\rho_0,
\end{equation}
and the positive upper and lower bounds \eqref{positive-global} hold uniformly
on $[0,T_*)$.

Suppose, by contradiction, that $T_*<\infty$. The uniform
bounds \eqref{uniform-global-final} and \eqref{positive-global}
provide fixed numbers $B,m,M$ satisfying the hypotheses of
Lemma~\ref{lem:uniform-restart} for every $t_0<T_*$. The lemma
supplies a common restart time $\tau(B,m,M)>0$. Restarting at
$t_0\in(T_*-\tau/2,T_*)$ and invoking uniqueness extends the
solution beyond $T_*$, contrary to maximality. Consequently,
\[
T_*=\infty.
\]

The estimates \eqref{H1-global} and \eqref{positive-global} therefore hold
for all time, and the global strong solution asserted in
Theorem \ref{thm:global-small} is obtained.
\end{proof}

\section{Long-Time Stability: Proof of Theorem \ref{thm:longtime}}
\label{sec:longtime}

The global energy bounds give integrated dissipation of all three
spatial gradients.  To deduce decay at every sufficiently large time,
we verify that their squared $L^2$ norms have integrable time
derivatives.  The latter point uses the equations and is not a purely
formal consequence of the space-time dissipation bound.

\begin{proof}[Proof of Theorem \ref{thm:longtime}]
Write $U=(v-\bar v,u,\theta-\bar\theta)$.
Combining \eqref{L2-bound}, \eqref{vx-L2}, \eqref{ux-L2}, and
\eqref{thetax-L2}, now on $(0,\infty)$, gives
\begin{equation}\label{longtime-global-estimate}
\begin{aligned}
&\sup_{t\ge0}\|U(t)\|_{H^1}^2
+\int_0^\infty
\Bigl(\|v_x(t)\|^2+\|u_x(t)\|^2+\|u_{xx}(t)\|^2
+\|\theta_x(t)\|^2+\|\theta_{xx}(t)\|^2\Bigr)
\,\mathrm{d}t
\le C\varepsilon_0^2.
\end{aligned}
\end{equation}
In particular, the one-dimensional interpolation inequality
$\|h\|_{L^\infty}^2\le2\|h\|\|h_x\|$ implies
\begin{equation}\label{longtime-infty-deriv}
\int_0^\infty
\bigl(\|u_x(t)\|_{L^\infty}^2
+\|\theta_x(t)\|_{L^\infty}^2\bigr)\,\mathrm{d}t
\le C\varepsilon_0^2.
\end{equation}

\medskip
\noindent\textit{Step 1. Integrability of the time derivatives.}

By the global time-derivative estimate \eqref{global-time-regularity},
\begin{equation}\label{longtime-temporal}
(u_t,\theta_t)\in L^2(0,\infty;L^2(\mathbb R)).
\end{equation}
Moreover, differentiating the continuity equation $v_t=u_x$ with respect
to $x$ gives
\[
(v_x)_t=u_{xx}\in L^2(0,\infty;L^2(\mathbb R)),
\]
where the last inclusion follows from \eqref{longtime-global-estimate}.

\medskip
\noindent\textit{Step 2. Decay of the gradients.}
Define, for $t\ge0$,
\[
F_v(t)=\|v_x(t)\|^2,\qquad
F_u(t)=\|u_x(t)\|^2,\qquad
F_\theta(t)=\|\theta_x(t)\|^2.
\]
These nonnegative functions belong to $L^1(0,\infty)$ by
\eqref{longtime-global-estimate}. We justify their time derivatives
at the regularity level of the strong solution, rather than
formally differentiating an $L^2$-valued spatial derivative.
On each finite time interval $[0,T]$ we have
\[
 v_x\in H^1(0,T;L^2),\qquad (v_x)_t=u_{xx}\in L^2(0,T;L^2)
\]
by the continuity equation. For the other two components,
\[
 u_x,\theta_x\in L^2(0,T;H^1),\qquad
 (u_x)_t=\partial_xu_t,\quad
 (\theta_x)_t=\partial_x\theta_t
 \quad\text{in }L^2(0,T;H^{-1}),
\]
where \eqref{longtime-temporal} is used. Applying the Hilbert-triple
chain rule for $H^1(\mathbb R)\hookrightarrow L^2(\mathbb R)
\hookrightarrow H^{-1}(\mathbb R)$ gives absolutely continuous
representatives of $F_u,F_\theta$ on $[0,T]$. The identities hold
for almost every $t>0$:
\begin{equation}\label{longtime-norm-derivatives}
\begin{aligned}
F_v'(t)&=2\int_{\mathbb R}v_x u_{xx}\,\mathrm{d}x,\\
F_u'(t)&=2\langle \partial_xu_t,u_x\rangle_{H^{-1},H^1}
         =-2\int_{\mathbb R}u_{xx}u_t\,\mathrm{d}x,\\
F_\theta'(t)&=2\langle\partial_x\theta_t,\theta_x\rangle_{H^{-1},H^1}
         =-2\int_{\mathbb R}\theta_{xx}\theta_t\,\mathrm{d}x.
\end{aligned}
\end{equation}
Here the last two equalities are distributional integration by parts
in space; they require no $L^2$ estimate of $u_{xt}$ or
$\theta_{xt}$. By Cauchy--Schwarz in space and time,
\begin{align*}
\int_0^\infty |F_v'(t)|\,\mathrm{d}t
&\le 2\|v_x\|_{L^2_tL^2_x}\|u_{xx}\|_{L^2_tL^2_x}<\infty,\\
\int_0^\infty |F_u'(t)|\,\mathrm{d}t
&\le 2\|u_{xx}\|_{L^2_tL^2_x}\|u_t\|_{L^2_tL^2_x}<\infty,\\
\int_0^\infty |F_\theta'(t)|\,\mathrm{d}t
&\le 2\|\theta_{xx}\|_{L^2_tL^2_x}\|\theta_t\|_{L^2_tL^2_x}<\infty.
\end{align*}
Thus $F_v,F_u,F_\theta\in W^{1,1}(0,\infty)$. Each has a finite
limit as $t\to\infty$, and since each is also integrable and
nonnegative, that limit must be zero. This proves
\eqref{longtime-gradients}.

In fact, integrating \eqref{longtime-norm-derivatives} from $t$ to
infinity, now using the proved zero limits, gives the explicit
\emph{dissipation-tail bounds}
\begin{equation}\label{longtime-gradient-tails}
\begin{aligned}
\|v_x(t)\|_2^2
&\le 2\left(\int_t^\infty\|v_x(s)\|_2^2\,\mathrm{d}s\right)^{1/2}
       \left(\int_t^\infty\|u_{xx}(s)\|_2^2\,\mathrm{d}s\right)^{1/2},\\
\|u_x(t)\|_2^2
&\le 2\left(\int_t^\infty\|u_{xx}(s)\|_2^2\,\mathrm{d}s\right)^{1/2}
       \left(\int_t^\infty\|u_t(s)\|_2^2\,\mathrm{d}s\right)^{1/2},\\
\|\theta_x(t)\|_2^2
&\le 2\left(\int_t^\infty\|\theta_{xx}(s)\|_2^2\,\mathrm{d}s\right)^{1/2}
       \left(\int_t^\infty\|\theta_t(s)\|_2^2\,\mathrm{d}s\right)^{1/2}.
\end{aligned}
\end{equation}
These estimates hold for the continuous representatives at each
$t\ge0$ and furnish another direct verification of gradient decay.
Their right-hand sides depend on the solution's own dissipative tails,
so they yield no universal algebraic rate in time.

\medskip
\noindent\textit{Step 3. Uniform convergence and $L^p$ convergence.}
For each scalar component $f$ of $U$, the one-dimensional inequality
\[
\|f(t)\|_{L^\infty}^2\le2\|f(t)\|\|f_x(t)\|
\]
and \eqref{longtime-global-estimate}--\eqref{longtime-gradients}
imply $\|f(t)\|_{L^\infty}\to0$.
This proves \eqref{longtime-uniform}. Finally, for $2<p<\infty$,
\[
\|f(t)\|_{L^p}^p
\le\|f(t)\|_{L^\infty}^{p-2}\|f(t)\|^2\to0,
\]
which yields \eqref{longtime-Lp}; the case $p=\infty$ was already
established. The proof is complete.
\end{proof}

\begin{remark}[No inference of full $L^2$ decay from dissipation]
The proof above obtains limits of $\|v_x(t)\|_2$,
$\|u_x(t)\|_2$, and $\|\theta_x(t)\|_2$ through the
$W^{1,1}(0,\infty)$ regularity of their squared norms. It does
\emph{not} infer that the limiting relative energy vanishes.
On the whole line, the relative-energy identity and gradient decay
by themselves do not provide the missing low-frequency control;
there is no global Poincar\'e inequality in this setting.
This distinction explains the additional hypotheses used for full
$H^1$ convergence below.
\end{remark}

\begin{remark}[Absence of a decay rate]
The argument provides convergence without a quantitative rate.
In particular, the global $H^1$ bound and integrated dissipation in
\eqref{longtime-global-estimate} alone do not control the low spatial
frequencies strongly enough to establish decay of the full $L^2$ norm
or an algebraic time-decay estimate.  Such refinements require
additional information and are not claimed here. Related large-time
results for one-dimensional real and polytropic gases were previously
studied in \cite{Jiang1994b,Jiang1999}; the present theorem is a
consequence of the energy bounds in the setting of this paper, rather
than a claim of priority for long-time stabilization of real gases.
\end{remark}

\section[Low-Frequency Analysis and a Conditional H1 Decay Criterion]{Low-Frequency Analysis and a Conditional $H^1$ Decay Criterion}
\label{sec:lowfrequency}

Theorem~\ref{thm:longtime} establishes decay of the spatial derivatives and
uniform convergence to equilibrium, but does not establish decay of the
nonlinear solution in $L^2(\mathbb R)$. This section separates two
statements: unconditional decay for the \emph{constant-coefficient
linearization}, and a conditional criterion for decay of the full nonlinear
solution. Neither statement is presented as an unconditional nonlinear
$H^1$ theorem.

\subsection{The constant-coefficient linearized problem}

Write $w=v-\bar v$, $\vartheta=\theta-\bar\theta$, and define the
positive constants and equilibrium derivatives
\[
 \alpha=-p_v(\bar v,\bar\theta)>0,\quad
 \beta=p_\theta(\bar v,\bar\theta),\quad
 c_*=e_\theta(\bar v,\bar\theta)>0,\quad
 \nu=\frac{\mu}{\bar v}>0,\quad
 d=\frac{\kappa}{\bar v c_*}>0,\quad
 \gamma=\frac{\bar\theta\beta}{c_*}.
\]
The linearized temperature formulation of \eqref{NS} about
$(\bar v,0,\bar\theta)$ is
\begin{equation}\label{linearized-system}
 \begin{cases}
 w_t-u_x=0,\\
 u_t-\alpha w_x+\beta\vartheta_x=\nu u_{xx},\\
 \vartheta_t+\gamma u_x=d\vartheta_{xx}.
 \end{cases}
\end{equation}
The third equation is obtained by linearizing the temperature equation,
not by replacing the internal energy with a constant specific heat away
from equilibrium. In particular, the same linearization applies near a
mechanically stable van der Waals equilibrium as discussed in
Section~\ref{sec:constitutive-models}, after the local smooth extension is used;
only the values of the constitutive derivatives at the equilibrium
enter this calculation.

\begin{proposition}[Linearized $L^2$ and $H^1$ decay]
\label{prop:linearized-decay}
Let \eqref{assumption_pe}, \eqref{thermo-relation}, and
\eqref{mechanical-stability} hold. For every initial perturbation
$W_0=(w_0,u_0,\vartheta_0)\in H^1(\mathbb R)$,
the solution $W=(w,u,\vartheta)$ of \eqref{linearized-system} satisfies
\begin{equation}\label{linearized-H1-vanish}
 \lim_{t\to\infty}\|W(t)\|_{H^1}=0.
\end{equation}
If in addition $W_0\in L^1(\mathbb R)$, then for constants $c,C>0$
depending only on the equilibrium coefficients,
\begin{align}
 \|W(t)\|_{L^2}
 &\le C(1+t)^{-1/4}\|W_0\|_{L^1}
      +Ce^{-ct}\|W_0\|_{L^2},\label{linearized-L2-rate}\\
 \|W_x(t)\|_{L^2}
 &\le C(1+t)^{-3/4}\|W_0\|_{L^1}
      +Ce^{-ct}\|(W_0)_x\|_{L^2}.
      \label{linearized-H1-rate}
\end{align}
These are estimates for the linearized system \eqref{linearized-system}.
\end{proposition}

\begin{proof}
Take the spatial Fourier transform, using hats for transformed
quantities. At a frequency $\xi\in\mathbb R$,
\begin{equation}\label{linearized-Fourier-system}
 \begin{cases}
 \widehat w_t=i\xi\widehat u,\\
 \widehat u_t=i\xi\alpha\widehat w-i\xi\beta\widehat\vartheta
                    -\nu\xi^2\widehat u,\\
 \widehat\vartheta_t=-i\xi\gamma\widehat u
                    -d\xi^2\widehat\vartheta.
 \end{cases}
\end{equation}
Set $\lambda_*=c_*/\bar\theta>0$, so that
$\lambda_*\gamma=\beta$. (The symbol $\lambda_*$ is used here to
avoid confusion with the entropy perturbation introduced in
Section~\ref{sec:entropy-frequency}.)
The symmetrized Fourier energy
\[
 E_0=\alpha|\widehat w|^2+|\widehat u|^2
                       +\lambda_*|\widehat\vartheta|^2
\]
obeys the exact identity
\begin{equation}\label{linearized-basic-Fourier}
 \frac{\mathrm{d}E_0}{\mathrm{d}t}
 =-2\nu\xi^2|\widehat u|^2
  -2\lambda_*d\xi^2|\widehat\vartheta|^2.
\end{equation}
The missing direct dissipation for $\widehat w$ is recovered by a
frequency-dependent compensating term, in the spirit of standard
hyperbolic--parabolic compensating energy methods
(see, for example, \cite{Kawashima1981}). Define
\[
 K(t,\xi)=\operatorname{Re}
       \bigl(i\xi\widehat w\overline{\widehat u}\bigr).
\]
Direct differentiation of \eqref{linearized-Fourier-system} yields
\begin{equation}\label{linearized-cross-identity}
 \frac{\mathrm{d}K}{\mathrm{d}t}
 =-\xi^2|\widehat u|^2+\alpha\xi^2|\widehat w|^2
   -\beta\xi^2\operatorname{Re}
            (\widehat w\overline{\widehat\vartheta})
   -\nu\xi^2K.
\end{equation}
For a sufficiently small fixed $\eta>0$, introduce
\begin{equation}\label{linearized-compensated-energy}
 \mathcal E(t,\xi)
 = E_0(t,\xi)-\frac{2\eta}{1+\xi^2}K(t,\xi).
\end{equation}
Because $|\xi|/(1+\xi^2)\le1/2$, the energy
$\mathcal E$ is uniformly equivalent to
$|\widehat w|^2+|\widehat u|^2+|\widehat\vartheta|^2$
for $\eta$ sufficiently small.

For completeness, put $q(\xi)=\xi^2/(1+\xi^2)$.
Combining \eqref{linearized-basic-Fourier} and
\eqref{linearized-cross-identity}, we get
\begin{align*}
 \frac{\mathrm{d}\mathcal E}{\mathrm{d}t}
 ={}&-2\nu\xi^2|\widehat u|^2
      -2\lambda_*d\xi^2|\widehat\vartheta|^2\\
 &+2\eta q(\xi)\bigl(
       |\widehat u|^2-\alpha|\widehat w|^2
       +\beta\operatorname{Re}
          (\widehat w\overline{\widehat\vartheta})
       +\nu K\bigr).
\end{align*}
The mixed terms are estimated by Young's inequality. For every
$\varepsilon>0$,
\begin{align*}
 q(\xi)|\widehat w\widehat\vartheta|
 &\le \varepsilon q(\xi)|\widehat w|^2
             +C_\varepsilon\xi^2|\widehat\vartheta|^2,\\
 q(\xi)|\xi\widehat w\widehat u|
 &\le \varepsilon q(\xi)|\widehat w|^2
             +C_\varepsilon\xi^2|\widehat u|^2.
\end{align*}
The second inequality applies to the term $q(\xi)K$, since
$|K|\le|\xi|\,|\widehat w|\,|\widehat u|$.
Here we used $q(\xi)\le\xi^2$ and $q(\xi)\le1$.
First choose $\varepsilon>0$ so small that
$\varepsilon(|\beta|+\nu)\le\alpha/4$. Then choose $\eta>0$
small enough that the energy $\mathcal E$ remains uniformly
equivalent to
$|\widehat w|^2+|\widehat u|^2+|\widehat\vartheta|^2$, that
$2\eta q(\xi)|\widehat u|^2$ is absorbed by the viscous term
(using $q(\xi)\le\xi^2$), and that the terms
$2\eta C_\varepsilon\xi^2|\widehat u|^2$ and
$2\eta C_\varepsilon\xi^2|\widehat\vartheta|^2$ are absorbed by
the viscous and thermal dissipations, respectively.
With these choices, the negative term
$-2\eta\alpha q(\xi)|\widehat w|^2$ absorbs the remaining
$w$-terms. Thus, for some $c_1>0$,
\[
 \frac{\mathrm d\mathcal E}{\mathrm dt}
 \le -c_1\left(
 q(\xi)|\widehat w|^2+
 \xi^2|\widehat u|^2+
 \xi^2|\widehat\vartheta|^2\right).
\]
Since $q(\xi)\le\xi^2$ and $\mathcal E$ is uniformly equivalent
to the squared Fourier norm, the right-hand side controls
$q(\xi)\mathcal E$. Therefore there is $c_0>0$ such that
\begin{equation}\label{linearized-hypocoercive}
 \frac{\mathrm{d}\mathcal E}{\mathrm{d}t}
 +c_0\frac{\xi^2}{1+\xi^2}\mathcal E\le0.
\end{equation}
Consequently, uniformly in $\xi$ and $t\ge0$,
\begin{equation}\label{linearized-symbol-estimate}
 |\widehat W(t,\xi)|
 \le C\exp\!\left(-c\frac{\xi^2}{1+\xi^2}t\right)
                 |\widehat W_0(\xi)|.
\end{equation}
For every $\xi\neq0$ the exponential tends to zero as $t\to\infty$.
Since it is bounded by a constant, the dominated convergence theorem
applied to $(1+\xi^2)|\widehat W_0(\xi)|^2$ proves
\eqref{linearized-H1-vanish} for arbitrary $H^1$ initial data.

If $W_0\in L^1$, then $\|\widehat W_0\|_{L^\infty}
\le C\|W_0\|_{L^1}$. Split the Fourier integrals at $|\xi|=1$.
On $|\xi|\le1$, the exponential in
\eqref{linearized-symbol-estimate} is bounded by $e^{-c'\xi^2t}$,
while on $|\xi|>1$ it is bounded by $e^{-c't}$.
In particular,
\begin{align*}
 \int_{|\xi|\le1}|\widehat W(t,\xi)|^2\,\mathrm{d}\xi
 &\le C(1+t)^{-1/2}\|W_0\|_{L^1}^2,\\
 \int_{|\xi|\le1}\xi^2|\widehat W(t,\xi)|^2\,\mathrm{d}\xi
 &\le C(1+t)^{-3/2}\|W_0\|_{L^1}^2.
\end{align*}
The high-frequency part is exponentially small in the appropriate
$L^2$ or $H^1$ norm. Plancherel's theorem now gives
\eqref{linearized-L2-rate}--\eqref{linearized-H1-rate}.
\end{proof}

\subsection[What suffices for the nonlinear H1 limit]{What suffices for the nonlinear $H^1$ limit}

For a vector-valued function $f$, write
\[
 \|f\|_{\dot H^{-s}}^2
   :=\int_{\mathbb R}|\xi|^{-2s}|\widehat f(\xi)|^2\,\mathrm{d}\xi,
 \qquad 0<s<\tfrac12,
\]
whenever the right-hand side is finite.

\begin{proposition}[Conditional full nonlinear $H^1$ stability]
\label{prop:conditional-H1}
Let $U(t)=(v-\bar v,u,\theta-\bar\theta)(t)$ be the global nonlinear
strong solution from Theorem~\ref{thm:global-small}. Suppose, in
addition, that for some $0<s<\tfrac12$ one has the \emph{uniform-in-time}
low-frequency estimate
\begin{equation}\label{conditional-negative-bound}
 \sup_{t\ge0}\|U(t)\|_{\dot H^{-s}}\le M<\infty.
\end{equation}
Then
\begin{equation}\label{conditional-H1-decay}
 \lim_{t\to\infty}\|U(t)\|_{H^1}=0.
\end{equation}
A sufficient, stronger alternative to \eqref{conditional-negative-bound}
is a \emph{uniform-in-time} bound
\begin{equation}\label{conditional-L1-bound}
 \sup_{t\ge0}\|U(t)\|_{L^1}<\infty.
\end{equation}
Neither \eqref{conditional-negative-bound} nor
\eqref{conditional-L1-bound} is asserted to follow merely from the
initial $H^1$ and thermodynamic hypotheses of
Theorem~\ref{thm:global-small}.
\end{proposition}

\begin{proof}
For $U(t)\in\dot H^{-s}\cap\dot H^1$, H\"older's inequality on the
Fourier side gives
\begin{align}\label{negative-interpolation}
 \|U(t)\|_{L^2}
 &\le C\|U(t)\|_{\dot H^{-s}}^{1/(s+1)}
                 \|U_x(t)\|_{L^2}^{s/(s+1)}\\
 &\le C M^{1/(s+1)}
                 \|U_x(t)\|_{L^2}^{s/(s+1)}.\nonumber
\end{align}
The gradient factor tends to zero by
\eqref{longtime-gradients}; hence $\|U(t)\|_{L^2}\to0$.
Adding \eqref{longtime-gradients} proves
\eqref{conditional-H1-decay}.

To verify that \eqref{conditional-L1-bound} implies
\eqref{conditional-negative-bound}, split the homogeneous negative
norm into $|\xi|\le1$ and $|\xi|>1$. Since $s<1/2$,
\[
 \int_{|\xi|\le1}|\xi|^{-2s}|\widehat U(t,\xi)|^2\,\mathrm{d}\xi
 \le C_s\|U(t)\|_{L^1}^2,
 \qquad
 \int_{|\xi|>1}|\xi|^{-2s}|\widehat U(t,\xi)|^2\,\mathrm{d}\xi
 \le C\|U(t)\|_{L^2}^2.
\]
The first quantity is uniformly bounded by
\eqref{conditional-L1-bound}, and the second by
\eqref{H1-global}. This proves the implication.
\end{proof}

\begin{remark}[Why the low-frequency assumption matters]
The family $f_L(x)=L^{-1/2}\phi(x/L)$, where
$0\not\equiv\phi\in C_c^\infty(\mathbb R)$ and $L\to\infty$,
satisfies
\[
 \|f_L\|_{L^2}=\|\phi\|_{L^2},\qquad
 \|\partial_x f_L\|_{L^2}=L^{-1}\|\phi'\|_{L^2}\to0,
 \qquad\|f_L\|_{L^\infty}\to0.
\]
Thus bounded $L^2$ mass, vanishing gradients and uniform convergence
are insufficient, \emph{without information from the evolution}, to
deduce $L^2$ convergence. In contrast, the Fourier estimate
\eqref{linearized-symbol-estimate} supplies exactly the missing
low-frequency control for the linearized flow. Extending that
mechanism to the nonlinear problem while deriving a uniform negative
norm or an equivalent low-frequency estimate solely from suitable
\emph{initial-data} hypotheses remains a separate task. In
particular, assuming $U_0\in L^1$ is not by itself a proof of the
uniform-in-time bound \eqref{conditional-L1-bound}. No algebraic rate
for the nonlinear solution is claimed in this section.
\end{remark}

\section[Low-Frequency Duhamel Analysis and Entropy-Coordinate Stabilization]{Low-Frequency Duhamel Analysis and Entropy-Coordinate Stabilization}
\label{sec:entropy-frequency}
This section develops the low-frequency mechanism needed for the
nonlinear $H^1$ stabilization result. We first establish an abstract
Duhamel lemma for constant-coefficient linear semigroups. We then
introduce the Gibbs entropy as a dependent variable, derive the exact
entropy-coordinate evolution system, and verify the source-term
integrability needed to apply the lemma. This yields a conditional
one-remainder criterion for general pressure laws; under the additional
local entropy-affine identity in Theorem~\ref{thm:entropy-affine-H1},
the pressure remainder vanishes and the full $H^1$ convergence follows.
The general-pressure criterion remains conditional because the required
remainder integrability is not established from the basic small-data
bounds alone.

\subsection{A Low-Frequency Duhamel Lemma}\label{subsec:duhamel-lemma}
The following functional-analytic lemma applies to any
constant-coefficient linear generator whose Fourier semigroup satisfies
the bound in \eqref{linearized-symbol-estimate}; the nonlinear forcing
assumptions must be verified separately for the system under study.

\begin{lemma}[Derivative forcing and integrable forcing at low frequency]
\label{lem:duhamel-criterion}
Let $\mathcal L$ be a constant-coefficient differential operator
realized on $H^{-1}(\mathbb R)^m$, with
$H^1(\mathbb R)^m\subset D(\mathcal L)$ and with the restriction
$\mathcal L:H^1(\mathbb R)^m\to H^{-1}(\mathbb R)^m$ continuous.
Let $S(t)$ be the strongly continuous semigroup on $H^{-1}(\mathbb R)^m$
generated by this realization of $\mathcal L$, and suppose that it is a
Fourier-multiplier semigroup with matrix symbol $M(t,\xi)$ satisfying
\begin{equation}\label{abstract-duhamel-symbol}
 \|M(t,\xi)\|\le C\exp\!\left(-c\frac{\xi^2}{1+\xi^2}t\right)
 \qquad(t\ge0,\ \xi\in\mathbb R).
\end{equation}
Suppose $X\in C([0,\infty);H^1(\mathbb R)^m)$ satisfies
$X_t=\mathcal LX+\partial_xF+G$ in distributions, where
\[
 F\in L^2(0,\infty;L^2(\mathbb R)^m),\qquad
 G\in L^1(0,\infty;L^2(\mathbb R)^m),\qquad
 \|X_x(t)\|_2\longrightarrow0.
\]
Then $\|X(t)\|_{H^1}\to0$ as $t\to\infty$.
\end{lemma}
\begin{proof}
We first make precise the mild formulation. Since
$F\in L^2_{\mathrm{loc}}L^2$ and $G\in L^1_{\mathrm{loc}}L^2$,
the forcing $\partial_xF+G$ belongs to
$L^1(0,T;H^{-1})$ for each finite $T$.
The multiplier estimate in \eqref{abstract-duhamel-symbol} gives
uniform boundedness of $S(t)$ on $H^{-1}$, while strong continuity
is part of the semigroup hypothesis. Since
$X\in C([0,T];H^1)$ and
$\mathcal L:H^1\to H^{-1}$ is continuous,
$t\mapsto\mathcal L X(t)$ is continuous with values in $H^{-1}$.
Together with the local integrability of $\partial_xF+G$, the
distributional equation therefore gives
$X\in W^{1,1}(0,T;H^{-1})$. The standard variation-of-constants
formula for the generator $\mathcal L$ on $H^{-1}$ then yields
\begin{equation}\label{abstract-duhamel-formula}
 X(t)=S(t)X(0)+\int_0^tS(t-r)\partial_xF(r)\,\mathrm{d}r
                    +\int_0^tS(t-r)G(r)\,\mathrm{d}r.
\end{equation}
One can also obtain this identity by applying finite-frequency
projections, solving the resulting Fourier equations and passing
to the limit in $C([0,T];H^{-1})$. In particular, the equality
may be projected onto $|\xi|\le1$ as an $L^2$ identity:
on this fixed frequency region the $H^{-1}$ and $L^2$ norms
are equivalent. Hence the subsequent low-frequency estimates
apply to the given solution, not merely to a formal Duhamel series.

Let $P_{\mathrm{lo}}$ and $P_{\mathrm{hi}}$ denote the Fourier
projections onto $|\xi|\le1$ and $|\xi|>1$, respectively.
At high frequencies, Plancherel gives
\begin{equation}\label{duhamel-highfreq}
 \|P_{\mathrm{hi}}X(t)\|_2\le\|X_x(t)\|_2\longrightarrow0.
\end{equation}
For the low-frequency initial-data term, the multiplier bound
and dominated convergence yield
\[
 \|P_{\mathrm{lo}}S(t)X(0)\|_2\longrightarrow0;
\]
indeed $M(t,\xi)\to0$ for every $\xi\ne0$ and is uniformly
bounded, while the single frequency $\xi=0$ is negligible.

We next examine the derivative forcing. On $|\xi|\le1$ the
multiplier estimate implies
\begin{equation}\label{duhamel-early-operator}
 \|P_{\mathrm{lo}}S(\tau)\partial_x f\|_2
 \le C(1+\tau)^{-1/2}\|f\|_2,
 \qquad \tau\ge0,
\end{equation}
since $|\xi|e^{-c'\xi^2\tau}\le C(1+\tau)^{-1/2}$
on that interval. Fix $T>0$ and take $t>2T$.
The part of the $F$ convolution from $0$ to $T$ is bounded by
\[
 \left\|P_{\mathrm{lo}}\int_0^T S(t-r)\partial_xF(r)\,\mathrm{d}r\right\|_2
 \le C(1+t-T)^{-1/2}\int_0^T\|F(r)\|_2\,\mathrm{d}r
 \longrightarrow0.
\]
For the tail, Fourier-side Cauchy--Schwarz gives, for $|\xi|\le1$,
\begin{align}\label{duhamel-fourier-L2}
 &\left|\int_T^t i\xi M(t-r,\xi)\widehat F(r,\xi)\,\mathrm{d}r\right|^2
 \notag\\[-1mm]
 &\quad\le C\left(\int_T^t\xi^2e^{-2c'\xi^2(t-r)}\,\mathrm{d}r\right)
               \int_T^\infty|\widehat F(r,\xi)|^2\,\mathrm{d}r
 \le C\int_T^\infty|\widehat F(r,\xi)|^2\,\mathrm{d}r.
\end{align}
Indeed, for $\xi\ne0$,
\[
 \int_T^t\xi^2e^{-2c'\xi^2(t-r)}\,\mathrm{d}r
 =\frac{1-e^{-2c'\xi^2(t-T)}}{2c'}\le\frac1{2c'},
\]
and for $\xi=0$ the integral is zero. Thus the constant is
uniform in both $t$ and the low frequency $\xi$.
Plancherel now shows that, \emph{uniformly in $t>T$},
\begin{equation}\label{duhamel-tail-L2}
 \left\|P_{\mathrm{lo}}\int_T^t
 S(t-r)\partial_xF(r)\,\mathrm{d}r\right\|_2
 \le C\|F\|_{L^2(T,\infty;L^2)}.
\end{equation}
First let $t\to\infty$ for fixed $T$; then let $T\to\infty$.
The right side of \eqref{duhamel-tail-L2} tends to zero.
Thus the entire low-frequency derivative-source contribution
vanishes in $L^2$.

For the $G$ convolution, the same early/late split applies.
For each fixed $r\in[0,T]$, strong $L^2$ stability of
$P_{\mathrm{lo}}S(t-r)$ follows from the Fourier multiplier
bound. The integrand is dominated by $C\|G(r)\|_2$,
so Bochner dominated convergence gives
\[
 \left\|P_{\mathrm{lo}}\int_0^T S(t-r)G(r)\,\mathrm{d}r\right\|_2
 \longrightarrow0.
\]
The tail is bounded by
\[
 \left\|P_{\mathrm{lo}}\int_T^t S(t-r)G(r)\,\mathrm{d}r\right\|_2
 \le C\int_T^\infty\|G(r)\|_2\,\mathrm{d}r\longrightarrow0
 \quad\text{as }T\to\infty,
\]
uniformly for $t\ge T$. More explicitly, the four terms in
\eqref{abstract-duhamel-formula} give, for every fixed $T>0$,
\begin{equation}\label{duhamel-limsup-tail}
 \limsup_{t\to\infty}\|P_{\mathrm{lo}}X(t)\|_2
 \le C\left(
 \|F\|_{L^2(T,\infty;L^2)}
 +\|G\|_{L^1(T,\infty;L^2)}\right).
\end{equation}
The right-hand side tends to zero as $T\to\infty$; hence
$\|P_{\mathrm{lo}}X(t)\|_2\to0$.
Together with \eqref{duhamel-highfreq}, this gives
$\|X(t)\|_2\to0$, and the assumed gradient decay then gives
$\|X(t)\|_{H^1}\to0$.
\end{proof}

\begin{remark}[Why the forcing spaces are different]
The derivative on $F$ is essential in the $L^2$-in-time forcing
estimate: it provides the $|\xi|$ factor in
\eqref{duhamel-fourier-L2} and removes the low-frequency singularity.
An arbitrary nondivergence forcing in $L^2_tL^2_x$ is not
covered by the lemma. The nondivergence term $G$ is instead
assumed in $L^1_tL^2_x$.
\end{remark}

\subsection{Exact Entropy Balance and the Linearized Semigroup}

Write
\[
 w=v-\bar v,\qquad h=s(v,\theta)-\bar s,
 \qquad \bar s=s(\bar v,\bar\theta),\qquad X=(w,u,h).
\]
Since $s_\theta=e_\theta/\theta>0$, the local map
$(v,\theta)\mapsto(v,s)$ has a $C^2$ inverse
$\theta=\Theta(v,s)$ near the equilibrium.  In that neighborhood,
$X\in H^1$ if and only if
$U=(v-\bar v,u,\theta-\bar\theta)\in H^1$, with
uniformly equivalent norms along the small global solution.
Indeed, the local mean-value theorem, the inverse map, and
\[
 h_x=p_\theta(v,\theta)w_x
          +\frac{e_\theta(v,\theta)}{\theta}\theta_x
\]
give constants $0<c<C$, independent of $t\ge0$, such that
\begin{equation}\label{entropy-H1-coordinate-equivalence}
 c\|U(t)\|_{H^1}\le\|X(t)\|_{H^1}
                         \le C\|U(t)\|_{H^1}.
\end{equation}
The same bounded first derivatives, together with
\eqref{global-time-regularity}, show that on every finite
interval $[0,T]$
\begin{equation}\label{entropy-time-regularity}
 h_t=p_\theta(v,\theta)v_t
       +\frac{e_\theta(v,\theta)}{\theta}\theta_t
          \in L^2(0,T;L^2),\qquad
 X\in C([0,T];H^1),\quad X_t\in L^2(0,T;L^2).
\end{equation}
These statements follow by the Sobolev chain rule and require no
second-order spatial derivatives of the thermodynamic map.
Moreover, the Gibbs relation and the temperature equation imply the
\emph{exact} entropy identity
\begin{equation}\label{entropy-exact-balance}
 s_t=\partial_x\!\left(\frac{\kappa\theta_x}{v\theta}\right)
   +\frac{\kappa\theta_x^2}{v\theta^2}
   +\frac{\mu u_x^2}{v\theta}.
\end{equation}
Indeed, the Gibbs relation and $v_t=u_x$ give
$\theta s_t=e_t+pu_x=(\kappa\theta_x/v)_x+
\mu u_x^2/v$; the identity
$\theta^{-1}Q_x=(Q/\theta)_x+Q\theta_x/\theta^2$
with $Q=\kappa\theta_x/v$ gives \eqref{entropy-exact-balance}.
By $s_x=p_\theta v_x+(e_\theta/\theta)\theta_x$,
\begin{equation}\label{entropy-flux-gradient}
 \frac{\kappa\theta_x}{v\theta}
 =\frac{\kappa}{v e_\theta(v,\theta)}
       \bigl(h_x-p_\theta(v,\theta)w_x\bigr).
\end{equation}

Put
\[
 \beta=p_\theta(\bar v,\bar\theta),\quad
 c_*=e_\theta(\bar v,\bar\theta)>0,\quad
 b_*=\frac{\bar\theta\beta}{c_*},\quad
 a_*=-p_v(\bar v,\bar\theta)+\frac{\bar\theta\beta^2}{c_*}>0,
\]
\[
 \nu=\frac\mu{\bar v},\qquad
 d=\frac\kappa{\bar v c_*}>0.
\]
These are the derivatives of the pressure with respect to the
\emph{entropy} coordinates:
\begin{equation}\label{entropy-pressure-derivatives}
 \left.\frac{\partial p}{\partial s}\right|_v=b_*,
 \qquad
 \left.\frac{\partial p}{\partial v}\right|_s=-a_*.
\end{equation}
For instance,
$(\partial p/\partial v)_s=p_v-p_\theta(s_v/s_\theta)
=p_v-\bar\theta\beta^2/c_*$ at the equilibrium.
The linearized system for $X$ is
\begin{equation}\label{entropy-linearized-system}
 \begin{cases}
 w_t-u_x=0,\\
 u_t-a_*w_x+b_*h_x=\nu u_{xx},\\
 h_t=d(h_{xx}-\beta w_{xx}).
 \end{cases}
\end{equation}
It is exactly the constant linearized system
\eqref{linearized-system} after the invertible change
$h=\beta w+(c_*/\bar\theta)\vartheta$. To check that the
semigroup estimate transfers with no additional spectral
hypothesis, put
\[
 T_*:=\begin{pmatrix}
 1&0&0\\0&1&0\\\beta&0&c_*/\bar\theta
 \end{pmatrix},\qquad
 X=T_*(w,u,\vartheta)^{\mathsf T}.
\]
Since $c_*>0$ and $\bar\theta>0$, the matrix $T_*$ is
invertible, and the constant-coefficient Fourier generators obey
$A_s(\xi)=T_*A(\xi)T_*^{-1}$. Thus
$M_s(t,\xi)=T_*M(t,\xi)T_*^{-1}$, and the multiplier estimate
\eqref{linearized-symbol-estimate} transfers with a constant
multiplied by at most $\|T_*\|\|T_*^{-1}\|$.
Consequently its semigroup $S_s(t)$ enjoys the Fourier estimate
\begin{equation}\label{entropy-symbol-bound}
 |\widehat{S_s(t)Z}(\xi)|
 \le C\exp\!\left(-c\frac{\xi^2}{1+\xi^2}t\right)
       |\widehat Z(\xi)|.
\end{equation}
We can therefore apply the low-frequency argument of
Lemma~\ref{lem:duhamel-criterion} to $X$ whenever
$X_t=\mathcal L_s X+\partial_xF_s+G_s$ with
$F_s\in L^2_tL^2_x$ and $G_s\in L^1_tL^2_x$.

\subsection{Characterization of the Entropy-Affine Pressure Class}

The entropy-affine hypothesis is a restrictive structural condition,
not another name for a general Gibbs-compatible equation of state.
It can be checked directly in the original thermodynamic variables,
without explicitly solving for the inverse temperature map.

\begin{proposition}[Characterization of the entropy-affine pressure class]
\label{prop:entropy-affine-characterization}
Let $p,e$ satisfy \eqref{assumption_pe}--\eqref{thermo-relation} on
an open, connected neighborhood $\mathcal N$ of the reference state,
with $e_\theta>0$. Then \eqref{entropy-affine-pressure} holds
throughout $\mathcal N$ for constants $a,b$ if and only if
\begin{equation}\label{entropy-affine-derivative-test}
 \frac{\theta p_\theta(v,\theta)}{e_\theta(v,\theta)}=b,
 \qquad
 p_v(v,\theta)-b p_\theta(v,\theta)=-a
 \quad\text{on }\mathcal N.
\end{equation}
In this case the mechanical-stability condition at the equilibrium is
precisely
\begin{equation}\label{entropy-affine-mechanical-test}
 a>\frac{b^2e_\theta(\bar v,\bar\theta)}{\bar\theta}.
\end{equation}
In particular, this inequality implies $a>0$.
\end{proposition}
\begin{proof}
Differentiating \eqref{entropy-affine-pressure} in $(v,\theta)$,
and using $s_v=p_\theta$ and
$s_\theta=e_\theta/\theta$, gives
$p_v=-a+b p_\theta$ and
$p_\theta=b e_\theta/\theta$. This proves
\eqref{entropy-affine-derivative-test}. Conversely, if
\eqref{entropy-affine-derivative-test} holds, the Gibbs identities imply
\[
 \mathrm{d}\bigl[p(v,\theta)+a(v-\bar v)
             -b(s(v,\theta)-\bar s)\bigr]=0
 \qquad\text{on }\mathcal N.
\]
Since $\mathcal N$ is connected, the bracket is constant there;
evaluating it at the equilibrium gives that constant as $\bar p$.
This is \eqref{entropy-affine-pressure}. Finally,
$p_v=-a+b^2e_\theta/\theta$ in this class, so
$p_v(\bar v,\bar\theta)<0$ is equivalent to
\eqref{entropy-affine-mechanical-test}.
\end{proof}

\begin{remark}[The ideal polytropic gas is not entropy-affine]
\label{rem:ideal-not-entropy-affine}
For $p=R\theta/v$ and $e=C_v\theta$ with $R,C_v>0$,
Gibbs' relation gives
$s=C_v\log\theta+R\log v+\text{constant}$.
At fixed $v$, this yields
\[
 \left.\frac{\partial p}{\partial s}\right|_v
 =\frac{p}{C_v},\qquad
 \left.\frac{\partial^2p}{\partial s^2}\right|_v
 =\frac{p}{C_v^2}>0.
\]
Hence $p$ is not affine in $(v,s)$ on any open thermodynamic
neighborhood. Theorem~\ref{thm:entropy-affine-H1} therefore does
\emph{not} include the standard ideal-gas law, even though the general
existence and gradient-decay theorems do include it. The explicit
model in Remark~\ref{rem:entropy-explicit-model} shows that the
entropy-affine class is nevertheless nonempty and can have
state-dependent specific heat. This comparison concerns the reach
of the present \emph{proof}; it does not assert that ideal-gas
solutions fail to converge in $H^1$ by other methods.
\end{remark}

\subsection{A Single Pressure-Remainder Criterion}

Define the thermodynamically natural Taylor remainder
\begin{equation}\label{entropy-pressure-remainder}
 \mathcal R_s(v,\theta)
 :=p(v,\theta)-\bar p+a_*(v-\bar v)
                    -b_*\bigl(s(v,\theta)-\bar s\bigr).
\end{equation}
Because the constitutive functions and the local entropy map are
$C^2$, the equilibrium derivative of $\mathcal R_s$ vanishes and
\begin{equation}\label{entropy-pressure-quadratic}
 |\mathcal R_s(v,\theta)|\le C(|w|^2+|h|^2)
\end{equation}
on the fixed thermodynamic neighborhood.
Combining the mass and momentum equations with
\eqref{entropy-exact-balance} and \eqref{entropy-flux-gradient}
gives the exact system
\begin{equation}\label{entropy-duhamel-system}
 X_t=\mathcal L_sX+\partial_xF_s+G_s,
\end{equation}
where
\begin{align}
 F_s={}&\left(0,\left(\frac\mu v-\nu\right)u_x-\mathcal R_s,
    \frac\kappa{v e_\theta}(h_x-p_\theta w_x)
                    -d(h_x-\beta w_x)\right),
                  \label{entropy-F}\\
 G_s={}&\left(0,0,
       \frac{\kappa\theta_x^2}{v\theta^2}
                    +\frac{\mu u_x^2}{v\theta}\right).
                  \label{entropy-G}
\end{align}
We spell out the coefficient estimates behind the flux bound.
Along the global small solution, $(v,\theta)$ remains in a fixed
compact thermodynamic neighborhood of $(\bar v,\bar\theta)$.
Smoothness of the constitutive functions and the definitions of
$\nu$ and $d$ therefore give
\[
 \left|\frac{\mu}{v}-\nu\right|\le C|w|,
 \qquad
 \left|\frac{\kappa}{v e_\theta}-d\right|
 +\left|\frac{\kappa p_\theta}{v e_\theta}-d\beta\right|
 \le C\bigl(|w|+|\theta-\bar\theta|\bigr).
\]
Consequently, using
$h_x=p_\theta w_x+(e_\theta/\theta)\theta_x$ and the uniform
small-data bound $\|U(t)\|_\infty\le C\varepsilon_0$,
\[
 \begin{aligned}
 \left\|\left(\frac{\mu}{v}-\nu\right)u_x\right\|_2
 &\le C\|w\|_\infty\|u_x\|_2,\\
 \left\|\left(\frac{\kappa}{v e_\theta}-d\right)h_x
 -\left(\frac{\kappa p_\theta}{v e_\theta}-d\beta\right)w_x
 \right\|_2
 &\le C\|U\|_\infty
       \bigl(\|v_x\|_2+\|\theta_x\|_2\bigr).
 \end{aligned}
\]
These estimates justify the following pointwise-in-time bounds,
with constants uniform in $t$:
\begin{equation}\label{entropy-flux-pointwise-bound}
 \|F_s(t)\|_2
 \le C\bigl(\|u_x(t)\|_2+\|v_x(t)\|_2
       +\|\theta_x(t)\|_2+\|\mathcal R_s(t)\|_2\bigr).
\end{equation}
\begin{equation}\label{entropy-source-pointwise-bound}
 \|G_s(t)\|_2
 \le C\bigl(\|u_x(t)\|_\infty\|u_x(t)\|_2
       +\|\theta_x(t)\|_\infty\|\theta_x(t)\|_2\bigr).
\end{equation}
In particular,
\begin{equation}\label{entropy-flux-source-bounds}
 \begin{split}
 &F_s\in L^2(0,\infty;L^2(\mathbb R))
    \quad\text{if }\mathcal R_s\in L^2(0,\infty;L^2(\mathbb R)),\\
 &G_s\in L^1(0,\infty;L^2(\mathbb R))\quad\text{without extra assumptions}.
 \end{split}
\end{equation}
Indeed, the first inclusion follows by integrating
\eqref{entropy-flux-pointwise-bound} in time and using the
$L^2_tL^2_x$ control of all three gradients from Section~4.
The second follows from \eqref{entropy-source-pointwise-bound},
Cauchy--Schwarz in time and $u_x,\theta_x\in L^2_tH^1_x$:
\[
 \|G_s\|_{L^1_tL^2_x}
 \le C\bigl(\|u_x\|_{L^2_tH^1_x}^2
             +\|\theta_x\|_{L^2_tH^1_x}^2\bigr)<\infty.
\]
These are the full nonlinear source terms; the variable viscosity
and thermal diffusion coefficients have not been discarded.
For clarity, all terms in \eqref{entropy-duhamel-system} also have
the distributional regularity needed to invoke the abstract
Duhamel formula. On each $[0,T]$, \eqref{entropy-time-regularity}
gives $X\in C_tH^1_x$ and $X_t\in L^2_tL^2_x$; the
constant-coefficient operator maps $H^1_x$ into $H^{-1}_x$,
while $\partial_xF_s+G_s\in L^1(0,T;H^{-1})$ by
\eqref{entropy-flux-source-bounds}. Hence the exact balance
\eqref{entropy-duhamel-system} is valid as an $H^{-1}$-valued
distributional evolution and its mild representation is precisely
\eqref{abstract-duhamel-formula}, with $\mathcal L=\mathcal L_s$.

\begin{proposition}[One-remainder nonlinear $H^1$ criterion]
\label{prop:entropy-pressure-criterion}
Under the hypotheses of Theorem~\ref{thm:global-small}, assume in
addition that the \emph{entropy-coordinate pressure remainder} obeys
\begin{equation}\label{entropy-remainder-condition}
 \mathcal R_s(v,\theta)\in
  L^2(0,\infty;L^2(\mathbb R)).
\end{equation}
Then the solution satisfies
$\|(v-\bar v,u,\theta-\bar\theta)(t)\|_{H^1}\to0$ as
$t\to\infty$.
\end{proposition}
\begin{proof}
By \eqref{entropy-flux-source-bounds},
$F_s\in L^2_tL^2_x$ and $G_s\in L^1_tL^2_x$.
The Fourier estimate \eqref{entropy-symbol-bound} permits
application of the Duhamel argument in
Lemma~\ref{lem:duhamel-criterion}, while
$\|X_x(t)\|_2\to0$ follows from
Theorem~\ref{thm:longtime} and $s_x=p_\theta v_x+
(e_\theta/\theta)\theta_x$.
Hence $\|X(t)\|_{L^2}\to0$. The local inverse
$\theta=\Theta(v,s)$ has bounded derivatives in the fixed
neighborhood, so $\|U(t)\|_2\le C\|X(t)\|_2$.
Together with Theorem~\ref{thm:longtime}, this proves the assertion.
\end{proof}

\subsection{Proof of Theorem~\ref{thm:entropy-affine-H1} and a Non-Affine Example}

\begin{proof}[Proof of Theorem~\ref{thm:entropy-affine-H1}]
The local identity \eqref{entropy-affine-pressure} and
\eqref{entropy-pressure-derivatives} imply that
$a=a_*$ and $b=b_*$; indeed differentiating
\eqref{entropy-affine-pressure} at the equilibrium in $(v,s)$
coordinates uniquely determines these constants.
Since the global solution stays in $\mathcal N$, we have
$\mathcal R_s(v(t,x),\theta(t,x))\equiv0$.
This identity eliminates only the nonlinear pressure remainder,
not the coefficient-dependent flux terms in \eqref{entropy-F}
or the entropy production in \eqref{entropy-G}; their
integrability was verified above. Proposition~\ref{prop:entropy-pressure-criterion}
completes the proof.
\end{proof}

\begin{remark}[An entropy-affine law with nonlinear temperature dependence]
\label{rem:entropy-explicit-model}
Choose $b>0$, $A>b^2 e^{\bar\theta+b\bar v}$, and a bounded smooth
function $\psi$ on $(0,\infty)$ with
$\psi(\bar v)=0$, $\psi(v)=v-\bar v$ in a neighborhood of $\bar v$.
Write $z_0=\bar\theta+b\bar v$ and choose
$P_0>A\sup_{v>0}\psi(v)+b e^{z_0}$.  Define
\begin{align}\label{entropy-example-law}
 s(v,\theta)&=e^{\theta+bv},\notag\\
 p(v,\theta)&=P_0-A\psi(v)
                    +b(e^{\theta+bv}-e^{z_0}),\\
 e(v,\theta)&=(\theta-1)e^{\theta+bv}
        +(-P_0+b e^{z_0})v
        +A\int_{\bar v}^v\psi(y)\,\mathrm{d}y.\notag
\end{align}
Direct differentiation verifies
\begin{equation}\label{entropy-example-check}
 s_v=p_\theta=b e^{\theta+bv},\quad
 s_\theta=\frac{e_\theta}{\theta}=e^{\theta+bv}>0,
 \quad e_v=\theta p_\theta-p,
 \quad e_\theta=\theta e^{\theta+bv}>0.
\end{equation}
The pressure is globally positive by the choice of $P_0$, and the
mechanical stability derivative at the equilibrium is
\[
 p_v(\bar v,\bar\theta)=-A+b^2e^{z_0}<0.
\]
Near the equilibrium $\psi(v)=v-\bar v$, so
\[
 p(v,\theta)=P_0-A(v-\bar v)
          +b\bigl(s(v,\theta)-s(\bar v,\bar\theta)\bigr).
\]
Thus Theorem~\ref{thm:entropy-affine-H1} applies to this genuine
non-affine $\theta$-dependence, with state-dependent specific heat
$e_\theta(v,\theta)=\theta e^{\theta+bv}$.  This example is a
mathematically admissible constitutive law, not an identification with
the classical van der Waals law; the latter is not claimed to satisfy
the entropy-affine hypothesis.
\end{remark}

\begin{remark}[What remains open under the present estimates]
The exact entropy balance eliminates the non-divergence source term
proportional to $u_x$ that arises in the temperature-coordinate
formulation, so that this term is absent from the entropy-coordinate
Duhamel system. The remaining pressure remainder
$\mathcal R_s=O(|X|^2)$, however, has not been proved to lie in
$L^2_tL^2_x$ for a generic general pressure law with only small
$H^1$ data. Indeed the quadratic bound \eqref{entropy-pressure-quadratic}
gives, using the uniform $L^2$ bound,
\[
 \|\mathcal R_s(t)\|_2^2
 \le C\|X(t)\|_\infty^2\|X(t)\|_2^2
 \le C\varepsilon_0^2\|X(t)\|_\infty^2.
\]
The proved convergence $\|X(t)\|_\infty\to0$ gives no
integrability in time of its square. To see why the available
space-time estimates alone cannot close this gap, take
$0\ne\phi\in C_c^\infty(\mathbb R)$ and set
$f(t,x)=L(t)^{-1/2}\phi(x/L(t))$ with $L(t)=1+t$.
Then $\sup_t\|f(t)\|_2<\infty$,
$\int_0^\infty\|f_x(t)\|_2^2\,\mathrm{d}t<\infty$, and
$\|f(t)\|_\infty\to0$, but
\[
 \int_0^\infty\|f(t)^2\|_2^2\,\mathrm{d}t
 =\|\phi\|_4^4
   \int_0^\infty\frac{\mathrm{d}t}{1+t}=\infty.
\]
Multiplying $f$ by any fixed positive small factor makes its uniform
$H^1$ norm arbitrarily small without changing this divergence.
This is an illustration about the \emph{a priori norms}, not a
solution of the Navier--Stokes system or a counterexample to its
$H^1$ stability. Consequently Proposition~\ref{prop:entropy-pressure-criterion}
is \emph{conditional} for a general pressure law, while
Theorem~\ref{thm:entropy-affine-H1} is a theorem for its stated
constitutive subclass. No general nonlinear algebraic rate is
derived from these estimates.
\end{remark}

\section{Van der Waals Motivation and Compatible Constitutive Laws}
\label{sec:constitutive-models}

The following examples illustrate both the reach and the limits of
the general thermodynamic assumptions. The van der Waals model is a
primary motivation for the present analysis, rather than a decorative
example: its singular physical domain requires localization, and its
mechanically stable and unstable branches have to be distinguished.
We include these calculations to clarify how the constitutive
assumptions apply to the van der Waals model and to distinguish the
general existence results from the more restrictive entropy-affine
stability theorem.

\subsection{The van der Waals model: Gibbs compatibility, stability, and localization}
A physically motivated example is the classical van der Waals (vdW)
equation of state \cite{vanDerWaals1873}. In Lagrangian
specific-volume coordinates, it is expressed as
\begin{equation}\label{vdw-eos}
 p_{\mathrm{vdW}}(v,\theta)
   =\frac{R\theta}{v-b}-\frac{a}{v^2},
 \qquad
 e_{\mathrm{vdW}}(v,\theta)
   =E(\theta)-\frac{a}{v},
 \qquad v>b,\quad\theta>0,
\end{equation}
where $a,b,R>0$ and $E\in C^2((0,\infty))$ satisfies $E'(\theta)>0$.
The classical caloric law corresponds to $E(\theta)=C_v\theta$, but a
nonconstant specific heat is also allowed; for instance,
$E(\theta)=C_0\theta+\frac12 C_1\theta^2$ with $C_0>0$ and $C_1\ge0$.
A direct calculation shows that
\begin{equation}\label{vdw-thermodynamic}
 (e_{\mathrm{vdW}})_v=\frac{a}{v^2}
 =\theta(p_{\mathrm{vdW}})_\theta-p_{\mathrm{vdW}},
 \qquad
 (e_{\mathrm{vdW}})_\theta=E'(\theta)>0.
\end{equation}
In fact, the Gibbs relation is satisfied with entropy
$s_{\mathrm{vdW}}(v,\theta)
 =R\log(v-b)+\int_{\theta_*}^{\theta}E'(r)/r\,\mathrm{d}r$
(up to an additive constant), where $\theta_*>0$ is fixed.

The mechanical derivative is
\begin{equation}\label{vdw-mechanical}
 (p_{\mathrm{vdW}})_v(v,\theta)
 =-\frac{R\theta}{(v-b)^2}+\frac{2a}{v^3}.
\end{equation}
Hence the mechanical stability assumption at an equilibrium
$(\bar v,\bar\theta)$ with $\bar v>b$ is equivalent to
\begin{equation}\label{vdw-stable-region}
 \bar\theta>\frac{2a(\bar v-b)^2}{R\bar v^3}.
\end{equation}
For example, when $a=b=R=1$, one has
$(p_{\mathrm{vdW}})_v(5,1)=-93/2000<0$ whereas
$(p_{\mathrm{vdW}})_v(5,1/4)=3/8000>0$.
Thus the derivative is not globally negative in the physical state
space.  The second state lies on a mechanically unstable branch;
condition \eqref{vdw-stable-region}, rather than global monotonicity,
is used for the small-data global result.

The stronger entropy-affine theorem, by contrast, does \emph{not}
apply directly to the classical van der Waals law, even on its stable
branch. Indeed, the necessary test
\eqref{entropy-affine-derivative-test} would require
\begin{equation}\label{vdw-not-entropy-affine}
 \frac{\theta(p_{\mathrm{vdW}})_\theta}
      {(e_{\mathrm{vdW}})_\theta}
 =\frac{R\theta}{(v-b)E'(\theta)}=\text{constant}
\end{equation}
on an open neighborhood. For each fixed positive temperature the
right-hand side varies strictly with $v$, so this is impossible.
This exclusion concerns only the separate full-$H^1$ structural
theorem, not the general-law existence and gradient-decay results.

There is a domain issue that must be distinguished from the preceding
thermodynamic identities.  The pair \eqref{vdw-eos} is singular at
$v=b$ and is defined only for $v>b$, whereas the standing assumption
\eqref{assumption_pe} requires $p,e\in C^2((0,\infty)^2)$.
Consequently the unmodified vdW pair does \emph{not} literally satisfy
that assumption. Nevertheless, it admits a smooth extension which
agrees with \eqref{vdw-eos} on any sufficiently small prescribed
neighborhood of a fixed equilibrium $\bar v>b$. More precisely,
choose such a neighborhood $\mathcal N$ with
$v\ge v_->b$ throughout $\mathcal N$, and fix
$0<\delta<(v_--b)/2$. Choose
$\chi\in C^\infty((0,\infty))$ with $0\le\chi\le1$,
$\chi(v)=0$ for $v\le b+\delta$, and $\chi(v)=1$ for
$v\ge b+2\delta$. Then
$\mathcal N\subset\{v\ge b+2\delta\}$, so the extension below
agrees with the classical law throughout $\mathcal N$.  Define
\begin{equation}\label{vdw-extension-coefficients}
 A_\delta(v)=
 \begin{cases}
  \dfrac{1}{v+b}, &0<v\le b+\delta,\\[3pt]
  \dfrac{1-\chi(v)}{v+b}+\dfrac{\chi(v)}{v-b},
         &v>b+\delta,
 \end{cases}
 \qquad
 B_\delta(v)=\frac{1-\chi(v)}{(v+b)^2}+\frac{\chi(v)}{v^2}.
\end{equation}
Then $A_\delta,B_\delta\in C^\infty((0,\infty))$ and are positive.
Set, for all $v,\theta>0$,
\begin{equation}\label{vdw-smooth-extension}
 \begin{aligned}
 p_\delta(v,\theta)&=R\theta A_\delta(v)-aB_\delta(v),\\
 e_\delta(v,\theta)&=E(\theta)-\frac{a}{\bar v}
              +a\int_{\bar v}^{v}B_\delta(s)\,\mathrm{d}s.
 \end{aligned}
\end{equation}
This extension obeys, throughout $(0,\infty)^2$,
\begin{equation}\label{vdw-extension-identity}
 (e_\delta)_v=aB_\delta
 =\theta(p_\delta)_\theta-p_\delta,
 \qquad (e_\delta)_\theta=E'(\theta)>0.
\end{equation}
Moreover, on the entire region $v\ge b+2\delta$ the functions in
\eqref{vdw-smooth-extension} coincide \emph{exactly} with the classical
constitutive functions in \eqref{vdw-eos}; the integral in
$e_\delta$ has been normalized at $\bar v$ for this purpose.
Notice that a global sign restriction on $p_\delta$ is not required by
the present theorems.

For an equilibrium satisfying \eqref{vdw-stable-region}, the global
theorem applies to \eqref{vdw-smooth-extension}.  Its uniform small-data
$H^1$ estimate, together with the one-dimensional Sobolev embedding
(see \eqref{H1-global}--\eqref{Linfty-global}), ensures
$v(t,x)>b+2\delta$ for all time after the initial perturbation is
chosen sufficiently small.  Hence the resulting solution actually
solves the \emph{unmodified} vdW system \eqref{vdw-eos} along its whole
trajectory. The same solution also satisfies the uniform convergence
to equilibrium and spatial-gradient decay stated in
Theorem~\ref{thm:longtime}.
Similarly, initial data uniformly separated from $v=b$
can be treated locally by taking a shorter time interval so that the
solution remains in the matching region.  These observations use a
localization of the constitutive laws; they do not replace
\eqref{assumption_pe} by an unproved assertion of global regularity
for the singular vdW formula.

\subsection{A smooth nonmonotone law}
The hypotheses permit positive pressure with a nonconstant specific heat and a pressure derivative that changes sign away from the reference state. For example, on $(0,\infty)^2$ set
\begin{equation}\label{example-constitutive}
\begin{aligned}
 p(v,\theta)&=\frac{\theta}{v}
   +\theta\exp\!\bigl(-(v-3)^2\bigr)
   +\frac{\theta^2}{10(1+v)^2},\\
 e(v,\theta)&=\theta+\frac{\theta^2v}{10(1+v)}.
\end{aligned}
\end{equation}
These functions are smooth and satisfy $p(v,\theta)>0$ and
\[
 e_\theta(v,\theta)=1+\frac{\theta v}{5(1+v)}>0,
 \qquad
 e_v(v,\theta)=\frac{\theta^2}{10(1+v)^2}
 =\theta p_\theta(v,\theta)-p(v,\theta).
\]
The equilibrium $(\bar v,\bar\theta)=(1,1)$ satisfies
\[
 p_v(1,1)=-1+4e^{-4}-\frac1{40}<0,
 \qquad
 p_v\bigl(\tfrac52,1\bigr)
 =-\frac4{25}+e^{-1/4}-\frac8{1715}>0.
\]
Thus the small-data global theorem applies near $(1,0,1)$ although $p_v$ is not negative throughout the thermodynamic state space. The example is intended only to illustrate the scope of the assumptions; it is not a claim of a new physical equation of state.

\section{Concluding Discussion}\label{sec:conclusion}

The relative-energy and effective-gradient argument gives global
$H^1$ strong solutions near mechanically stable equilibria for
thermodynamically compatible pressure and internal-energy laws,
without requiring the specific heat to be constant. The resulting
space--time estimates imply decay of all first spatial derivatives,
and consequently $L^p(\mathbb R)$ relaxation for every $p>2$.
On the unbounded domain these estimates do not, by themselves,
control the low-frequency $L^2$ mass.

The entropy-coordinate formulation isolates a pressure remainder
whose additional $L^2_tL^2_x$ control suffices for full $H^1$
relaxation. That control holds identically in the locally
entropy-affine pressure subclass but has not been derived here for an
arbitrary compatible constitutive law. The structural subclass is
proper and excludes the classical ideal polytropic pressure law.
Accordingly, the results should be read as general-pressure
existence and gradient-decay theorems together with a more restricted
full-$H^1$ stabilization theorem, rather than as an unconditional
nonlinear $H^1$ decay theorem for all pressure laws.

\section*{Acknowledgments}
Qingsong Zhao was supported by the National Natural Science Foundation of China under Grant Number 12401281.

\end{document}